\documentclass{zaa}
\usepackage{color}

\usepackage{mathtools}
\usepackage{nicefrac}
\usepackage{enumerate}
\usepackage{url}
\usepackage{mathrsfs}
\renewcommand{\descriptionlabel}[1]
{\hspace*{0.5cm}\textit{#1}}
\newcommand{\R}{\mathbb{R}}
\newcommand{\N}{\mathbb{N}}
\newcommand{\K}{\mathcal{K}}
\newcommand{\LL}{\mathcal{L}}
\newcommand{\RR}{\mathcal{R}}
\newcommand{\T}{\mathcal{T}}
\newcommand{\CC}{\mathscr{C}}
\newcommand{\PP}{\mathscr{P}}
\newcommand{\HH}{\mathscr{H}}
\newcommand{\degg}{\textup{deg}}

\DeclareMathOperator*{\dist}{dist}
\DeclareMathOperator*{\divv}{div}

\newcommand{\lp}{L^p(\Omega)}
\newcommand{\wspr}{W^{s,p}(\R^n)}
\newcommand{\wsp}{W^{s,p}(\Omega)}
\newcommand{\wup}{W^{1,p}(\Omega)}
\newcommand{\wspc}{W^{s,p}_0(\Omega)}
\newcommand{\twsp}{\widetilde{W}^{s,p}(\Omega)}
\newcommand{\wupc}{W^{1,p}_0(\Omega)}
\newtheorem{teo}{Theorem}[section]
\newtheorem{lem}[teo]{Lemma}
\newtheorem{co}[teo]{Corollary}
\newtheorem{pro}[teo]{Proposition}
\theoremstyle{remark}
\newtheorem{re}[teo]{Remark}

\theoremstyle{definition}

\begin{document}
\title{Global bifurcation for fractional $p$-Laplacian\\
 and an application}
\author{Leandro M. Del Pezzo and Alexander Quaas}
\runtitle{Global bifurcation for fractional $p$-Laplacian} 
\runauthor{L. M. Del Pezzo and Alexander Quaas}.

\address{L. M. Del Pezzo
\hfill\break\indent CONICET and Departamento  de Matem{\'a}tica, FCEyN,
Universidad de Buenos Aires,
\hfill\break\indent Pabellon I, Ciudad Universitaria (1428),
Buenos Aires, Argentina.
\email{ldpezzo@dm.uba.ar}}
\address{A. Quaas
\hfill\break\indent Departamento de Matem\'atica, Universidad T\'ecnica 
Federico Santa Mar\'ia Casilla V-110, Avda. Espa\~na, 1680 -- 
Valpara\'iso, CHILE.
\email{alexander.quaas@usm.cl}}

\abstract{We prove the existence of an unbounded branch of solutions to the 
	non-linear non-local equation
	$$
		(-\Delta)^s_p u=\lambda |u|^{p-2}u + 
		f(x,u,\lambda) \quad\text{ in }\quad  \Omega,\quad u=0 
		\quad\text{ in }\quad \R^n\setminus\Omega ,
	$$
	bifurcating from the first eigenvalue. 
	Here $(-\Delta)^s_p$ denotes the fractional $p$-Laplacian
 	and  $\Omega\subset\R^n$  is a bounded regular domain. 
	The proof of the bifurcation results relies in computing the 
	Leray--Schauder degree by making an homotopy respect to $s$  (the 
	order of the fractional $p$-Laplacian) 
	and then to use results of local case 
	(that is $s=1$) found in \cite{delPino}. 
	Finally, we give some application to an existence result.}

\maketitle


\section{Introduction}
	In this paper, we study  Rabinowitz's global bifurcation type 
	result  form the first eigenvalue in a bounded domain of the  
	non-linear non-local operator called the fractional $p$-Laplacian 	
	operator, that is
	\begin{equation}\label{eq:opfint}
		(-\Delta)^s_p u=2\K(1-s)\mbox{ P.V.}\int_{\R^n}
		\frac{|u(x)-u(y)|^{p-2}(u(x)-u(y))}{|x-y|^{n+sp}}\, dy,
	\end{equation}
	where $\K$ is a constant depending on the dimension and $p$. 
	Observe that, this operator extends the fractional Laplacian ($p= 2$).
	
	\medskip
		
	More precisely, we prove the existence of an unbounded 
	branch of solutions to the non-linear non-local equation
	\begin{equation}\label{eq:oppt}
		\begin{cases}
			(-\Delta)^s_p u=
			\lambda |u|^{p-2}u + f(x,u,\lambda) &\text{ in }
			\Omega,\\
			u=0 &\text{ in } \mathbb{R}^n\setminus\Omega	,		
		\end{cases}		
	\end{equation}
	bifurcating from the first eigenvalue of the fractional 
	$p$-Laplacian assuming that $f$ is  
	$o(|u|^{p-2}u)$ near zero and $\Omega\subset\mathbb{R}^n$ 
	is a bounded regular domain. 

	\medskip

	Bifurcation and global bifurcation are basic principles in 
	mathematical analysis that can be established using, for example, 
	implicit function theorem or degree theory and, in some simple
	situation,  sub and super solution method,  i.e Perron's 
	method. In particular, bifurcation is used as a starting point to 
	prove existence of solution to ODE's and PDE's, 
	see for example \cite{Ku, RA1}. Some of the pioneer works related 
	with our method can be found in \cite{CR,RA3,RA2}.  
	Then many others generalization are established in different context 
	of local operator, see for instance 
	\cite{BR,BD,BEQ,delpino1, delPino,dra,dra1,Jorge,Ki}
	and the reference therein.

	\medskip

 	Fractional equations are nowadays classical in analysis, 
 	see for example \cite{S}. Fractional Laplacian 
 	have attracted much interest since they are connected with 
 	different applications and sometimes from the mathematical point 
 	of view the non-local character introduce difficulties that need 
 	some new approaches, see for instance \cite{DNPV,Sil} and 
 	the reference therein.

	\medskip	
	
	In \cite{CLM}, the fractional $p$-Laplacian is studied through 
	energy  and test function methods and it is used to obtain 
	H\"older extensions.
	See also \cite{BCF, BCF1}, where the authors consider a non-local
	``Tug-of-War" game, and \cite{IN}.

	\medskip

	Recently, existence and simplicity of the first eigenvalue in a 
	bounded domain for the fractional $p$-laplacian are obtained and also 
	some regularity result are established in 
	\cite{di2014local, FP,IMS,LL}. Some results of these works extend 
	the results of \cite{Anane} to the non-local case.
	
	\medskip
	
	In the process of writing this article, appearing 
	the following work \cite{IMS} where the authors, using  
	barrier arguments,  prove $C^\alpha$-regularity up to 
	the boundary for the weak solutions of a non-local non-linear 
	problem driven by the fractional $p$-Laplacian operator. This result 
	generalises the main result in \cite{RS}, where the case $p=2$ is studied.
	
	\medskip
	
	Thus, there is natural to ask if bifurcation occurs, this is even not
	known, as far as we know, for the case $p=2$,
	except for some related very recent
	results that can be found in \cite{MR3340666,FBS, PSY}.
	More precisely, in \cite{FBS}, the authors prove a multiplicity and 
	bifurcation result for the following problem
	\begin{equation}\label{eq:citas1}
		\begin{cases}
			-\mathscr{L}_{K}u=\lambda u+|u|^{2^\star-2}u &\text{ in }
			\Omega,\\
			u=0 &\text{ in } \mathbb{R}^n\setminus\Omega,
		\end{cases}
	\end{equation}
	where $n>2s$ and $2^\star=\nicefrac{2n}{n-sp}.$ Here $\mathscr{L}_K$
	is the non-local operator
	\[
		-\mathscr{L}_{K}u(x)=\int_{\mathbb{R}^n}
		(u(x+y)+u(x-y)-2u(x))K(y)\, dy, \quad x\in \mathbb{R}^n
	\] 
	whose model is given by the fractional Laplacian. They show that,
	in a suitable left neighborhood of any Dirichlet eigenvalue
	of $-\mathscr{L}_K,$ the number of trivial solution of \eqref{eq:citas1}
	is at least twice the multiplicity of the eigenvalue. In \cite{PSY},
	the authors extend the above bifurcation and multiplicity result to 
	the fractional $p-$Laplacian operator. Finally in \cite{MR3340666}, using variational method, the authors prove that the next problem admits at 
	least one non-trivial solution 
	\[
		\begin{cases}
	       (-\Delta)^{s}u(x)-\lambda u=\mu f(x,u) &\text{ in }\Omega,\\		
			u=0 &\text{ in }\mathbb{R}^n\setminus\Omega,\\
		\end{cases}
	\]
	where $n>2s,$ $f$ is a function satisfying suitable regularity and 
	growth conditions and 
	the parameters $\lambda$ and $\mu$ lie in a suitable range.

	\medskip

	In our approach, to establish the Rabinowitz's type of global 
	bifurcation result, we use Leray-Schauder degree that can be 
	computed  by making an homotopy respect to $s$ (the order of the 
	fractional $p$-Laplacian operator) and then use the homotopy 
	invariance of the Leray--Schauder degree to deduce that the degree 
	is the same as in the local case ($s=1$), i.e. the 
	$p$-Laplacian, which is already computed in \cite{delPino}. 
	Notice that in  \cite{delPino} similar ideas are used, where the  
	homotopy was done with respect to $p$ and the result were deduced
	from the (by now) classical case of the Laplacian. To do this 
	homotopy with respect to $s,$ we need as a starting point different 
	properties of the first eigenvalue in terms of $s$ up to 
	$s=1$, analogous properties to the ones that were obtained in 
	\cite{delPino}, but now with respect to $s$ not respect to $p$.
	
	Notice that one of our limiting procedures $s$ to $1$ are obtained 
	in the weak formulation with the help of some limiting properties 
	of the fractional Sobolev spaces already studied in \cite{bourgain}. 
	Moreover, in \cite{IN} this limiting procedure is done by viscosity 
	solution techniques for a very close related operator.
	
	\medskip
	
	Before stated our main theorem we will give the precisely assumption of
	the function $f\colon\Omega\times\R\times\R\to\R$:
	\begin{enumerate}
		\item $f$ satisfies a Carath\'eodory condition 
		in the first two variables;
		\item $f(x,t,\lambda)=o(|t|^{p-1})$ near $t=0,$ 
		uniformly a.e. with respect 
		to $x$ and uniformly with respect to $\lambda$ on bounded sets;
		\item There exists $q\in (1,p^\star_s)$ such that
			\[
				\lim_{|t|\to\infty}\dfrac{|f(x,t,\lambda)|}{|t|^{q-1}}=0
			\]
			uniformly a.e. with respect to $x$ and uniformly with 
			respect to $\lambda$ on bonded sets.
	\end{enumerate}
	Here $p_s^\star$ is the fractional critical 
			Sobolev exponent, that is
			\[
			p_s^\star\coloneqq
				\begin{cases}
					 \dfrac{np}{n-sp} &\text{ if } sp<n,\\
					 \infty  &\text{ if } sp\ge n.\\
				\end{cases}
				\]
	We denote by $\lambda_1(s,p)$ the first eigenvalue of following
	eigenvalue problem
	\begin{equation}
		\label{eq:epint}
		 \begin{cases}
			(-\Delta)_p^su=\lambda|u|^{p-2}u 
			&\text{ in } \Omega,\\
			u=0 &\text{ in } \R^n\setminus\Omega.
		\end{cases}	
	\end{equation}
	
	\medskip

	Our main result:

	\begin{teo}\label{teo:teo11}
		Let $\Omega\subset\R^n$ be a bounded domain with Lipschitz 
		boundary,
		$s\in(0,1),$ and $p\in(1,\infty).$
		The pair $(\lambda_1(s,p),0)$ is a bifurcation point of 
		\eqref{eq:oppt}. Moreover, there is a connected component 
		of the
		set of non-trivial weak solutions of \eqref{eq:oppt} in
		$\R\times\twsp$ whose closure contains $(\lambda_1(s,p),0)$
		and it is either unbounded or contains a pair $(\mu,0)$
		for some eigenvalue $\mu$ of \eqref{eq:epint} with
		$\mu>\lambda_1(s,p).$ 
	\end{teo}

	Notice that the ideas of the proof can be used for other problems. 
	As for example, a very closely 
	related problem such as bifurcation from infinity by the change of	 
	variable $ v=\nicefrac{u}{\|u\|_{\twsp}^2}$, for details see for example 
	\cite{dra1}. 
	
	\medskip

	Then, we use the above theorem for some application, more precisely, 
	we prove existence of a  non-trivial weak solution of the following 
	non-linear non-local problem
	\begin{equation} \label{eq:DD}
		\begin{cases}
			(-\Delta)^s_p u=g(u) &\text{ in } \Omega,\\
			u=0 &\text{ in } \R^n\setminus \Omega,\\
		\end{cases}
	\end{equation} 
	where $\nicefrac{g(t)}{|t|^{p-2}t}$ is bounded and crosses 
	the first eigenvalue. 
	
	\begin{teo}\label{teo:teo31}
	Let $g\colon\Omega\to\R$ continuous such that $g(0)=0$ and $g$ satisfies 
	\begin{enumerate}[\mbox{A}1.]
		\itemsep.5em 
		\item	$\dfrac{g(t)}{|t|^{p-2}t}$ is bounded;
		\item  $\underline{\lambda}\coloneqq\displaystyle\lim_{t\to0}
			{\dfrac{g(t)}{|t|^{p-2}t}<\lambda_1(s,p)<
			\displaystyle\liminf_{|t|\to\infty}\dfrac{g(t)}{|t|^{p-2}t}}.$ 
	\end{enumerate}
	Then there exists a non-trivial weak solution $u$ 
	of \eqref{eq:DD} such that $u$ has constant-sign in $\Omega$.
	\end{teo}

	For the prove of this existence result we need some 
	extra qualitative  properties of the branch of solutions in the above 
	theorem. Some of these properties come in some cases from the 
	study of the first eigenvalue of the fractional $p$-Laplacian with
	weights, see Section \ref{ffl}.
	
	\bigskip

	The paper is organized as follows. In Section \ref{Preli}, we review 
	some results of fractional Sobolev spaces and some properties of 
	the Leray-Schauder degree; in Section \ref{DP} we study the Dirichlet 
	problem with special interest in proving continuity in terms of $s$ 
	(see Lemma \ref{lem:contdeR} below);  in Section \ref{ffl} we study 
	the eigenvalue 
	problem with weights. In 
	addition, we establish the continuity of the eigenvalue respect to 
	$s$ that will help us to make the homotopy and then to compute the degree.  
	In Section \ref{bifurcation} we prove our main theorem.
	Finally, in Section \ref{aplica} we prove our existence results.

\section{Preliminaries}\label{Preli}

\subsection{Fractional Sobolev spaces}
	First, we briefly recall the definitions and some elementary
	properties of the fractional Sobolev spaces. We refer the reader to 
	\cite{Adams, DD, DNPV, Grisvard} for further reference 
	and for some of the proofs of the results in this subsection.

	\medskip

	Let $\Omega$ be an open set in $\mathbb{R}^n,$ $s\in(0,1)$ and
	$p\in [1,\infty).$ We define the fractional Sobolev space
	$\wsp$ as follows
	\[
		\wsp\coloneqq\left\{u\in L^p(\Omega)\colon
		\int_{\Omega}\int_{\Omega}
		\dfrac{|u(x)-u(y)|^p}{|x-y|^{n+ps}}\, dxdy<\infty
		\right\},
	\]
	endowed with the norm
	\[
		\|u\|_{\wsp}\coloneqq\left(\|u\|_{\lp}^p+|u|_{\wsp}^p
		\right)^{\frac1p},
	\]
	where
	\[
		\|u\|_{\lp}^p\coloneqq\int_\Omega |u(x)|^p\, dx 
		\quad\mbox{ and }\quad
		|u|_{\wsp}^p\coloneqq
		\int_{\Omega}\int_{\Omega}
		\dfrac{|u(x)-u(y)|^p}{|x-y|^{n+ps}}\, dxdy.
	\]

	A proof of the following proposition can be found in \cite{Adams,DD}.
	\begin{pro}
		Let $\Omega$ be an open set in $\mathbb{R}^n,$ $s\in(0,1)$ and
		$p\in [1,\infty).$ We have that
		\begin{itemize}
			\item $\wsp$ is a separable Banach space;
			\item If $1<p<\infty$ then $\wsp$ is reflexive.
		\end{itemize}
	\end{pro}

	We denote by $\wspc$ the closure of the space 
	$C_0^\infty(\Omega)$ of smooth functions with compact support 
	in $\wsp.$ We denote by $\twsp$ the space of all $u\in\wsp$ such that
	$\tilde{u}\in\wspr,$ where $\tilde{u}$ is the extension by zero of 
	$u.$ 

	\medskip

	The proofs of the next theorem is given in 
	\cite[Theorem 7.38]{Adams}. 
	\begin{teo}\label{teo:denso}
		For any $s\in(0,1)$ and $p\in(1,\infty),$ 
		the space $C_0^\infty(\mathbb{R}^n)$ is dense
		in $W^{s,p}(\mathbb{R}^n),$ that is 
		$W^{s,p}_0(\mathbb{R}^n)=W^{s,p}(\mathbb{R}^n).$  
	\end{teo}

	In the next result,  we show the explicit dependence of the 
	constant of \cite[Proposition 2.1]{DNPV} on $s$, that is needed for our propose.

	\begin{lem}\label{lem:inclu}
		Let $\Omega$ be an open set in $\mathbb{R}^n,$ 
		$p\in[1,\infty)$ and $0<s\le s^\prime<1.$
		Then 
		\begin{equation}
			\label{eq:a1}
			|u|_{\wsp}^p\le |u|_{W^{s^\prime,p}(\Omega)}^p
			+C(n,p)\left(\dfrac{1}{sp}-\dfrac{1}{s^\prime p}\right)
			\|u\|_{\lp}^p
		\end{equation}
		for any $u\in W^{s^\prime,p}(\Omega).$
	\end{lem}

	\begin{proof}
		 Let $u\in W^{s^\prime,p}(\Omega),$ then
			\begin{equation}\label{eq:a2}
				\begin{aligned}
					|u|_{\wsp}^p&=\int_\Omega\int_{\Omega}
					\dfrac{|u(x)-u(y)|^p}{|x-y|^{n+ps}}\, dxdy\\
								&=\int_\Omega\int_{A_y}
								\dfrac{|u(x)-u(y)|^p}{|x-y|^{n+ps}}\, dxdy
								+\int_\Omega\int_{\Omega\setminus A_y}
									\dfrac{|u(x)-u(y)|^p}{|x-y|^{n+ps}}
									\, dxdy
				\end{aligned}
			\end{equation}
		where $A_y=\Omega\cap\{x\in\R^n\colon |x-y|< 1\}.$
	
		Using that $s^\prime\ge s,$ we have that
		\begin{equation}\label{eq:a3}
			\int_\Omega\int_{A_y}
			\dfrac{|u(x)-u(y)|^p}{|x-y|^{n+ps}}\, dxdy
			\le \int_\Omega\int_{A_y}
			\dfrac{|u(x)-u(y)|^p}{|x-y|^{n+ps^\prime}}\, dxdy.
		\end{equation}

		On the other hand, we have that
		\begin{align*}
			&\int_\Omega\int_{\Omega\setminus A_y}
				\dfrac{|u(x)-u(y)|^p}{|x-y|^{n+ps}}\, dxdy=
				\int_\Omega\int_{\Omega\setminus A_y}
				\dfrac{|u(x)-u(y)|^p}{|x-y|^{n+ps^\prime}}
				|x-y|^{(s^\prime-s)p}\, dxdy\\
			=&\int_\Omega\int_{\Omega\setminus A_y}
				\dfrac{|u(x)-u(y)|^p}{|x-y|^{n+ps^\prime}}
				\left(|x-y|^{(s^\prime-s)p}-1\right)\, dxdy\\
			&\quad+\int_\Omega\int_{\Omega\setminus A_y}
				\dfrac{|u(x)-u(y)|^p}{|x-y|^{n+ps^\prime}}\, dxdy\\
			\le&2^{p-1}\int_\Omega\int_{\Omega\setminus A_y}
				\dfrac{|u(x)|^p + |u(y)|^p}{|x-y|^{n+ps^\prime}}
				\left(|x-y|^{(s^\prime-s)p}-1\right)\, dxdy\\
				&\quad+\int_\Omega\int_{\Omega\setminus A_y}
				\dfrac{|u(x)-u(y)|^p}{|x-y|^{n+ps^\prime}}\, 
				dxdy			
		\end{align*}
		Observe that 
		for any $x,y\in\Omega$ we have that
		$x\in \Omega\setminus A_y$ if only if $y\in \Omega\setminus A_x,$
		that is $\chi_{\Omega\setminus A_y}(x)=\chi_{\Omega\setminus A_x}(y).$
		Therefore
		\begin{align*}
			&\int_{\Omega}\int_{\Omega\setminus A_y}\dfrac{|u(x)|^p}{|x-y|^{n+sp}}
			\left(|x-y|^{(s^\prime-s)p}-1\right)\,dx dy=\\
			&\quad=\int_{\Omega}\int_{\Omega}\dfrac{|u(x)|^p}{|x-y|^{n+sp}}\chi_{\Omega\setminus A_y}(x)
			\left(|x-y|^{(s^\prime-s)p}-1\right)\,dx dy\\
			&\quad=\int_{\Omega}\int_{\Omega}\dfrac{|u(x)|^p}{|x-y|^{n+sp}}\chi_{\Omega\setminus A_x}(y)
			\left(|x-y|^{(s^\prime-s)p}-1\right)\,dx dy\\
			&\quad=\int_{\Omega}\int_{\Omega\setminus A_x}\dfrac{|u(x)|^p}{|x-y|^{n+sp}}
			\left(|x-y|^{(s^\prime-s)p}-1\right)\,dy dx.
		\end{align*}
		Thus		
		\begin{align*}
			&\int_\Omega\int_{\Omega\setminus A_y}
				\dfrac{|u(x)-u(y)|^p}{|x-y|^{n+ps}}\, dxdy\le\\
			&\le 2^{p}\int_\Omega\int_{\Omega\setminus A_x}
				\dfrac{|u(x)|^p}{|x-y|^{n+ps^\prime}}
				\left(|x-y|^{(s^\prime-s)p}-1\right)\, dxdy\\
				&\qquad+\int_\Omega\int_{\Omega\setminus A_y}
				\dfrac{|u(x)-u(y)|^p}{|x-y|^{n+ps^\prime}}\, 
				dxdy\\			
			&\le 2^{p}\int_\Omega |u(x)|^p\int_{\{|x-y|\ge1\}}
				\dfrac{|x-y|^{(s^\prime-s)p}-1}{|x-y|^{n+ps^\prime}}
				\, dydx\\
				&\qquad+\int_\Omega\int_{\Omega\setminus A_y}
				\dfrac{|u(x)-u(y)|^p}{|x-y|^{n+ps^\prime}}\, 
				dxdy\\	
			&\le 2^{p}\int_\Omega |u(x)|^p\int_{\{|z|\ge1\}}
				\dfrac{|z|^{(s^\prime-s)p}-1}{|z|^{n+ps^\prime}}
				\, dzdx\\
				&\qquad+\int_\Omega\int_{\Omega\setminus A_y}
				\dfrac{|u(x)-u(y)|^p}{|x-y|^{n+ps^\prime}}\, 
				dxdy\\					
			&\le 2^p\|u\|_{\lp}^p\int_{\{|z|\ge 1\}}
				\dfrac{|z|^{(s^\prime-s)p}-1}{|z|^{n+ps^\prime}}\ dz
			+\int_\Omega\int_{\Omega\setminus A_y}
				\dfrac{|u(x)-u(y)|^p}{|x-y|^{n+ps^\prime}}\, dxdy.
		\end{align*}
				
		Then
		\begin{equation}\label{eq:a4}
			\begin{aligned}
				\int_\Omega\int_{\Omega\setminus A_y}
				\dfrac{|u(x)-u(y)|^p}{|x-y|^{n+ps}}\, dydx\le&
				C(n,p)\left(\dfrac{1}{ps}-\dfrac{1}{ps^\prime}\right)
				\|u\|_{\lp}^p\\
				&+\int_\Omega\int_{\Omega\setminus A_y}
				\dfrac{|u(x)-u(y)|^p}{|x-y|^{n+ps^\prime}}\, dxdy.
			\end{aligned}
		\end{equation}
	
		Therefore, combining \eqref{eq:a2}, \eqref{eq:a3} and 
		\eqref{eq:a4},
		we get
		\[
			|u|_{\wsp}^p\le |u|_{W^{s^\prime,p}(\Omega)}^p
			+C(n,p)\left(\dfrac{1}{sp}-\dfrac{1}{s^\prime p}\right)
			\|u\|_{\lp}^p.
		\]
		The proof is now complete.
	\qed \end{proof}

	\begin{re}
		The space $W^{s^\prime,p}(\Omega)$ is continuously embedded in 
		$\wsp$ for any $0< s\le s^\prime < 1$ and $1\le p<\infty.$  
	\end{re}

	\begin{lem}
		\label{lem:poincare} Let $\Omega$ be an bounded open set in 
		$\mathbb{R}^n,$ $s\in(0,1)$ and $ p\in[1,\infty).$ 
		Then 
		\[
			\|u\|_{L^p(\Omega)}^p\le 
			\dfrac{sp|\Omega|^{\frac{sp}n}}
			{2\omega_n^{\frac{sp}{n}+1}}|u|_{\wspr}^p
		\]	
		for any $u\in \twsp.$ Here $\omega_{n}$ denotes $n$-dimensional 
		measure of the unit sphere $S^n.$
	\end{lem}

	\begin{proof}
		Let $u\in \twsp.$  Then
		\begin{align*}
			|u|_{\wspr}^p&=\int_{\Omega}\int_{\Omega}
			\dfrac{|u(x)-u(y)|^p}{|x-y|^{n+ps}}\, dxdy
			+2\int_\Omega\int_{\R^n\setminus\Omega}
			\dfrac{|u(x)|^p}{|x-y|^{n+ps}}\, dxdy\\
			&\ge 2\int_\Omega|u(x)|^p\int_{\R^n\setminus\Omega}
			\dfrac{1}{|x-y|^{n+ps}}\, dydx.
		\end{align*}

		Let $r=(\nicefrac{|\Omega|}{w_n})^{\nicefrac1n}.$ 
		Following the proof of Lemma 6.1 in \cite{DNPV},
		we get
		\begin{align*}
					\int_{\R^n\setminus\Omega}
							\dfrac{1}{|x-y|^{n+ps}}\, dydx
							&\ge
							\int_{\R^n\setminus B_{r}(x)}
							\dfrac{1}{|x-y|^{n+ps}}\, dydx
							=\omega_{n}
							\int_{r}^{\infty}\dfrac{d\rho}{\rho^{sp+1}}\\
							&=\dfrac{\omega_n}{sp}\dfrac{1}{r^{sp}}
		\end{align*} 
		which proves the lemma.
	\qed \end{proof}

	The proofs of the next two theorems are given in
	\cite[Proposition 2.2]{DNPV},  and \cite[Proposition 4.43]{DD},
	 respectively.

	\begin{teo}\label{teo:incluWspenW1p}
		Let $\Omega$ be an open set
		in $\R^n$ of class $C^{0,1}$ with bounded boundary 
		$s\in(0,1),$ and $p\in(1,\infty).$  Then,
		there exists a positive constant $C=C(n,s,p)$ such that
		\[
			\|u\|_{\wsp}\le C\|u\|_{\wup}\quad \forall u\in\wup .
		\]
		In particular, $\wup$ is continuously embedded in 
		$\wsp.$
	\end{teo}

	\begin{teo}\label{teo:incluar}
		Let $s\in(0,1),$  $p\in[1,\infty),$ and 
		$\Omega\subset\mathbb{R}^n$ be an open set with  Lipschitz
		boundary. Then $\wsp$ is continuously embedded in 
		$\wspr.$
	\end{teo}

	The proof of the following embedding theorem can be found in 
	\cite[Theorems 4.47]{DD}.

	\begin{teo}\label{teo:contincr} 
		Let $s\in(0,1)$ and $p\in(1,\infty).$ 
		Then we have the following continuous embeddings:
		\begin{align*}
			&W^{s,p}(\mathbb{R}^n)\hookrightarrow L^q(\mathbb{R}^n) 
			\qquad \mbox{ for all } 1\le q\le p_s^\star
			 &\mbox{ if } sp< n;\\
			 &W^{s,p}(\mathbb{R}^n)\hookrightarrow L^q(\mathbb{R}^n) 
			\qquad \mbox{ for all } 1\le q< \infty
			 &\mbox{ if } sp=n;\\
			&W^{s,p}(\mathbb{R}^n)\hookrightarrow 
			C^{0,\beta}_b(\mathbb{R}^n)
			\quad \mbox{ where } 
			\beta=s-\nicefrac{n}{p}, 
			&\mbox{ if } sp>n.
		\end{align*}
		Here $p_s^\star$ is the fractional critical Sobolev exponent, 
		that is
		\[
			p_s^\star\coloneqq
				\begin{cases}
					 \dfrac{np}{n-sp} &\text{ if } sp<n,\\
					 \infty  &\text{ if } sp\ge n.\\
				\end{cases}
		\]
	\end{teo}

	\begin{re} Note that $p_s^\star,$ as a function of $s,$ 
		is continuous in $(0,1]$ where $p_1^\star$ is the critical
		Sobolev exponent, i.e.  
		\[
			p_1^\star\coloneqq
				\begin{cases}
					 \dfrac{np}{n-p} &\text{ if } p<n,\\
					 \infty  &\text{ if } p\ge n,\\
				\end{cases}
		\]	
	\end{re}

	A proof of the next theorem can be found in \cite[Theorem 1]{Ms}.

	\begin{teo}\label{teo:mazya}
		Let $s\in(0,1),$ $p\in[1,\infty)$ and $sp<n.$ Then 
		there exists a constant $C=C(n,p)$ such that
		\[
			\|u\|_{L^{p_s^\star}(\R^n)}^p\le C\dfrac{s(1-s)}{(n-sp)^{p-1}}
			|u|_{\wspr}^p
		\]
		for all $u\in\wspr.$	
	\end{teo}

	By Theorem \ref{teo:incluar} and Theorem \ref{teo:contincr}, we have 
	the next result.

	\begin{co}\label{co:incont}
		Let $s\in(0,1),$  $p\in(1,\infty)$ and 
		$\Omega\subset\mathbb{R}^n$
		be an open set with Lipschitz boundary. The conclusions of
		Theorem \ref{teo:contincr} remain true if $\mathbb{R}^n$ is 
		replaced by $\Omega.$ 
	\end{co} 

	The following embedding theorem is established in 
	\cite[Theorem 4.58]{DD}. See also \cite{Adams}. 

	\begin{teo}\label{teo:compinc} 
		Let $\Omega\subset\mathbb{R}^n$ be
		a bounded open set with Lipschitz boundary, $s\in(0,1)$ and 
		$p\in[1,\infty).$ Then we have the following compact embeddings:
		\begin{align*}
			&\wsp\hookrightarrow L^q(\Omega) 
			\qquad \mbox{ for all } q\in[1,p_s^\star), 
			&\mbox{ if } sp\le n;\\
			&\wsp\hookrightarrow C^{0,\lambda}_b(\Omega) 
			\quad \mbox{ for all } 
			\lambda<s-\nicefrac{n}{p}, &\mbox{ if } sp>n.
		\end{align*}
	\end{teo}

	\begin{re} 
		Let $\Omega\subset\mathbb{R}^n$
		a bounded domain with Lipschitz boundary. 
		By the above theorem, we have that the 
		embedding of $\wsp$ into $L^p(\Omega)$ is compact for every
		$s\in(0,1)$ and for every $p\in(1,\infty).$
	\end{re}

	The next results are proven in  \cite[Corollaries 2 and 7]{bourgain}.

	\begin{teo}
		\label{teo:bbm1}	
		Let $\Omega$ be a smooth bounded domain in $\R^n,$ and 
		$p\in(1,\infty).$ Assume $u\in\lp,$ then
		\[
			\lim_{s\to 1^-}\K(1-s)|u|_{\wsp}^p=|u|_{\wup}^p 
		\]
		with
		\[
			|u|_{\wup}^p =\begin{cases}
			\displaystyle\int_{\Omega}|\nabla u|^p \, dx 
				&\text{ if } u\in\wup,\\
				\infty &\text{ if } u\notin\wup.
			\end{cases}	
		\]
		Here $\K$ depends only the $p$ and $n.$ 	
	\end{teo}

	\begin{re}\label{re:sopcomp} 
		Let $\Omega$ be a smooth bounded domain in $\R^n,$  
		$p\in(1,\infty)$ and $\phi\in C_0^\infty(\Omega).$ Then
		\begin{align*}
			|\phi|_{\wsp}^p&\le|\phi|_{\wspr}^p
			=|\phi|_{\wsp}^p + 
			2\int_{\Omega}\int_{\R^n\setminus\Omega}
			\dfrac{|\phi(x)|^p}{|x-y|^{n+sp}} 
			\, dx\\
			&\le 
			|\phi|_{\wsp}^p +
			\dfrac{C}{sp}
			\dfrac{\|\phi\|_{\lp}^p}{\dist(K,\partial\Omega)^{sp}},
		\end{align*}
		where $K$ is the support of $\phi$ and $C$ depends only of $n.$ 
		Then by Theorem \ref{teo:bbm1} we have
		\[
			\lim_{s\to 1^-}\K(1-s)|\phi|_{\wspr}^p=|\phi|_{\wup}^p. 
		\]
	\end{re}

	\begin{teo}
		\label{teo:bbm2}	
		Let $\Omega$ be a smooth bounded domain in $\R^n,$  
		$p\in(1,\infty)$ and $u_s\in\wsp$ for $s\in(0,1).$ 
		Assume that
		\[
			\int_{\Omega}u_s dx=0 \quad \text{ and }
			\quad (1-s)|u_s|_{\wsp}^p\le C
		\]
		for all $s\in(0,1).$ Then, there exists $u\in\wup$ and a subsequence 
		$\{u_{s_k}\}_{k\in\N}$ such that $s_k\to 1^-$ as 
		$k\to \infty,$
		\begin{align*}
			u_{s_k}\to u &\quad\mbox{strongly in } \lp,\\
			u_{s_k}\rightharpoonup u&\quad\mbox{weakly in }
			 W^{1-\varepsilon,p}(\Omega),
		\end{align*}
		for all $\varepsilon> 0.$ 
	\end{teo}
	
	Note that in the previous theorem, the assumption
	\[
			\int_{\Omega}u_s dx=0 \quad \forall s\in(0,1) 	
	\]
	can be replaced by $\{u_s\}_{s\in(0,1)}$ is bounded in $\lp.$
 
	\begin{re}
		\label{re:seminor(1-s)}	
		Let $0<s<s^\prime<1,$ and $1<p<\infty.$
		From the proof of the Lemma 2 and Corollary 7 in \cite{bourgain}, 
		it follows that
		\begin{equation}\label{eq:seminor}
			(1-s)|u|_{\wsp}^p\le
			2^{(1-s)p}\mbox{diam}(\Omega)^{(s^\prime-s)p}
			(1-s^\prime)|u|_{W^{s^\prime,p}(\R^n)}^p
		\end{equation}
		for all $u\in W^{s^\prime,p}(\R^n).$ Here
		$\mbox{diam}(\Omega)$ denotes the diameter of $\Omega.$ 
		See also \cite[Remark 6]{bourgain}. 
	
		Observe also that for any $u=\phi\in C_0^\infty(\Omega)$ 
		passing to the limit
		in \eqref{eq:seminor} as $s^\prime\to1$ and 
		using Theorem \ref{teo:bbm1},
		we get
		\[
			(1-s)|\phi|_{\wsp}^p\le\dfrac{2^{(1-s)p}
			\mbox{diam}(\Omega)^{(1-s)p}}{\K} |\phi|_{\wup}^p,
		\] 
		that is
		\[
			\K(1-s)|\phi|_{\wsp}^p\le 2^{(1-s)p}
			\mbox{diam}(\Omega)^{(1-s)p}|\phi|_{\wup}^p.
		\] 	
	\end{re}

	\begin{re}\label{re:mazya}
		Let $s_0\in(0,\min\{\nicefrac{n}{p},s\}),$ 
		$u\in \widetilde{W}^{s,p}(\Omega),$ $x\in\Omega$
		and  $B=B_{r}(x)$ with $r=\mbox{diam}(\Omega).$
		Then, by Theorem \ref{teo:mazya}, 
		there exists a constant $C= C(n,p)$ such that
		\begin{align*}
			&\|u\|_{L^{p_{s_0}^\star}(\R^n)}^p
			\le C\dfrac{s_0(1-s_0)}{(n-s_0p)^{p-1}}
			|u|_{W ^{s_0,p}(\R^n)}^p\\
			&= \dfrac{Cs_0}{(n-s_0p)^{p-1}}
			\left((1-s_0)|u|_{W ^{s_0,p}(B)}^p + 
			2(1-s_0)\int_{\Omega} \int_{\R^n\setminus B}
			\dfrac{|u(x)|^p}{|x-y|^{n+s_0p}}
			dxdy\right).
		\end{align*}
		By Remark \ref{re:seminor(1-s)},
		\[
			(1-s_0)|u|_{W ^{s_0,p}(B)}^p
			\le2^{(1-s_0)p}(4\mbox{diam}(\Omega))^{(s-s_0)p}
			(1-s)|u|_{\wspr}^p.
		\]
		On the other hand
		\[
			2(1-s_0)\int_{\Omega} \int_{\R^n\setminus B}
			\dfrac{|u(x)|^p}{|x-y|^{n+s_0p}}
			dxdy\le \dfrac{2^{1-s_0p}\omega_n}{\mbox{diam}(\Omega)^{s_0p} 
			\,s_0p}
			\int_{\Omega}|u(x)|^p dx.
		\]
	Then there exists a constant $C=C(n,p)$ such that
		\begin{align*}
			&\|u\|_{L^{p_{s_0}^\star}(\R^n)}^p
			\le\\
			& \dfrac{C}{(n-s_0p)^{p-1} \mbox{diam}(\Omega)^{s_0p}}
			\left(\mbox{diam}(\Omega)^{sp}(1-s)|u|_{W ^{s,p}(\R^n)}^p + 
			\dfrac1{s_0p}\int_{\Omega}|u(x)|^p
			dx\right).
		\end{align*}
	\end{re}

	Our last result gives a characterization of $\wspc.$ For the proof we 
	refer the reader to \cite[Corollary 1.4.4.5]{Grisvard}.

	\begin{teo}\label{teo:teouce}
		Let $\Omega\subset\R^n$ be bounded open set 
		with Lipschitz boundary, 
		$s\in(0,1]$ and $p\in(1,\infty).$ 
		If $s\neq\frac{1}{p}$ then
		\[
			\wspc=\twsp,
		\]
		Furthermore, when $0<s<\nicefrac1p$ we have
		\[
			\wspc=\wsp.
		\]  
	\end{teo}

	\begin{re}\label{re:extension} 
		$\twsp$ is a Banach space for the norm induced by $\wspr.$ 
		Moreover,
		if $\Omega\subset\R^n$ is a bounded open set 
		with Lipschitz boundary, 
		$s\in(0,1]$ and $p\in(1,\infty),$ then 
		$C_0^{\infty}(\Omega)$ is dense in $\twsp$ 
		and $\twsp\subset\wspc.$  
		See \cite[Theorem 1.4.2.2 and Corollary 1.4.4.10]{Grisvard}.
	\end{re}	

	\medskip

	For $s\in (0,1)$ and $p\in(1,\infty),$ we define
	the space $W^{-s,p'}(\Omega)$ ($\widetilde{W}^{-s,p'}(\Omega)$) as 
	the dual space of $\wspc$ ($\twsp$) where 
	$\nicefrac{1}{p'}+\nicefrac{1}{p}=1.$

	\subsection{Leray-Schauder degree}
		For the definition and some properties of Leray-Schauder degree,
		 for instance, see \cite{chang,RA1}.

		The proof of the next Leray-Schauder degree property is given in 
		\cite[Lemma 2.4]{delPino}. 

		\begin{lem}
			\label{lem:LSDP} Let $X,Y$ be Banach spaces with respective 
			norms $\|\cdot\|_X$	and $\|\cdot\|_Y.$ 
			Assume that $Y\subset X$ and that
			the inclusion $i\colon Y\to X$ is continuous. Let $\Omega_X,$ 
			$\Omega_Y$ be bounded open sets in $X$ and $Y,$  respectively,
			both containing $0,$ let $T\colon X\to Y$ be a completely
			continuous operator such that
			\[
				x-Tx\neq0\quad \forall x\in X \setminus\{0\}.
			\]
			Then
			\[
				\degg_X(I-i\circ T,\Omega_X,0)=\degg_Y(I-T\circ i,
				\Omega_Y,0).
			\]
		\end{lem}

\section{The Dirichlet problem}\label{DP}
	Let $\Omega$ be a smooth bounded domain in $\R^n,$ 
	and $p\in(1,\infty).$ We consider the
	operator 
	\begin{equation}
		\label{eq:op}
		\LL_{s,p}u\coloneqq
		\begin{cases}
				-\Delta_p u &\text{ if } s=1,\\
				(-\Delta)^s_p u &\text{ if } 0<s<1,	
		\end{cases}
	\end{equation}
	where $\Delta_p$ is the $p-$Laplace operator, that is
	\[
		\Delta_p u\coloneqq\divv(|\nabla u|^{p-2}\nabla u), 
	\]
	and $(-\Delta)^s_p$ is the fractional $p-$Laplace operator, that is
	\begin{equation}
		\label{eq:opf}
		(-\Delta)^s_p u=2\K(1-s)\mbox{P.V.}\int_{\R^n}
		\dfrac{|u(x)-u(y)|^{p-2}(u(x)-u(y))}{|x-y|^{n+sp}}\, dy
	\end{equation}
	with $\K$ is the constant of Theorem \ref{teo:bbm1}.

	For further details on the fractional $p-$Laplace operator,
	we refer to \cite{FP,LL} and references therein.

	\medskip

	It is well known that the Dirichlet problem 
	\begin{equation}
		\label{eq:app}
		\begin{cases}
			-\Delta_p u=h &\mbox{in }\Omega,\\
			u=0 &\mbox{on }\partial\Omega,
		\end{cases}
	\end{equation}
	has a unique weak solution for each $h\in W^{-1,p'}(\Omega),$ i.e.
	there exists a unique $u\in\wupc$ such that
	\[
		\int_{\Omega}|\nabla u(x)|^{p-2}\nabla u(x)\nabla\phi(x)\, dx=
		\langle h,\phi\rangle \quad\forall \phi\in C_0^\infty(\Omega),
	\]
	where $\langle\cdot,\cdot\rangle$ denotes the duality 
	pairing between $\wupc$ and $W^{-1,p'}(\Omega).$

	We also recall that the weak solution is the unique minimizer
	of the functional $J_{1,p}:\wupc\to\R$ given by
	\[
		J_{1,p}(v)=\dfrac1p|v|_{\wupc}^p-\langle h, v\rangle.
	\]
	See, for instance, \cite{struwe} and references therein.  
	
	\medskip

	Now, we study the Dirichlet problem for fractional 
	$p-$Laplace equation.

	\medskip

	Let $s\in(0,1),$ $p\in(1,\infty)$ and $h\in W^{-s,p'}(\Omega).$ 
	We say that $u\in\twsp$ is a weak solution of the Dirichlet problem
	\begin{equation}
		\label{eq:apf}
		\begin{cases}
			(-\Delta)^s_p u=h &\mbox{in } \Omega,\\
			u=0 &\mbox{in }\R^n\setminus\Omega, 
		\end{cases} 
	\end{equation}
	if
	\[
		\K(1-s)\mathcal{H}_{s,p}
		(u,v)=\langle h, v \rangle_s\quad 
		\forall v\in \twsp,
	\]
	where
	\begin{equation}
		\label{eq:H}
		\mathcal{H}_{s,p}(u,v)\coloneqq\int_{\R^n}\int_{\R^n}
		\dfrac{|u(x)-u(y)|^{p-2}(u(x)-u(y))}{|x-y|^{n+sp}}
		(v(x)-v(y))\, dydx,
	\end{equation}
	and $\langle\cdot,\cdot\rangle_s$ denotes the duality pairing between 
	$\twsp$ and $\widetilde{W}^{-s,p'}(\Omega).$

	\medskip

	It is clear that, the weak solutions are critical
	points of the functional $J_{s,p}:\twsp\to\R$ given by
	\[
		J_{s,p}(v)=\dfrac1p\K(1-s)|v|_{\wspr}^p-\langle h, v\rangle_s.
	\]

	Now, it is easy to see that $J_{s,p}$ 
	is bounded below, coercive, strictly convex and
	sequentially weakly lower semi continuous. 
	Then it has a unique critical point which is a global minimum.
	Therefore the Dirichlet problem \eqref{eq:apf} has a unique weak 
	solution.

	\medskip

	Thus, given $s\in(0,1]$ and $h\in \widetilde{W}^{-s,p'}(\Omega),$ the 
	Dirichlet problem
	\begin{equation}
		\label{eq:ap}
		\begin{cases}
			\LL_{s,p} u=h &\mbox{in } \Omega,\\
			u=0 &\mbox{in }\R^n\setminus\Omega, 
		\end{cases} 
	\end{equation}
	has a unique weak solution $u_{s,p,h}\in \twsp.$ Moreover, the 
	operator
	\begin{align*}
		\RR_{s,p}\colon \widetilde{W}^{-s,p'}(\Omega)&\to\twsp\\
					h&\to u_{s,p,h}
	\end{align*}
	is continuous. By the Rellich-Kondrachov theorem (case $s=1$)
	and  Theorem \ref{teo:compinc} (case $s\in(0,1)$), the restriction of
	$\RR_{s,p}$ to $L^{q^\prime}(\Omega)$ with $q\in(1, p^\star_s)$
	is a completely continuous operator, 
	that is  for every weakly convergent sequence 
	$\{h_k\}_{k\in\N}$ from $L^{q^\prime}(\Omega)$, 
	the sequence $\{\RR_{s,p}(h_k)\}_{k\in\N}$ is norm-convergent in 
	$\twsp.$

	\medskip

	Our next result show that the operator $\RR_{s,p}$ 
	is continuous with respect to $s$ and $h.$
	\begin{lem}
		\label{lem:contdeR} 
		Let  $p\in(1,\infty),$ $s_0\in(0,1),$
		and $1<q<p_{s_0}^\star.$ Then the operator
		\begin{align*}
			\RR_{p}\colon [s_0,1]\times L^{q^\prime}(\Omega)
			&\to L^{q}(\Omega)\\
			(s,h)&\to \RR_{s,p}(h)
		\end{align*}
		is completely continuous.
	\end{lem}

	\begin{proof}
		We start by proving that $\RR_p$ is compact.
	
		Let $\{(s_k,h_k)\}_{k\in\N}$ be a bounded sequence in
		$[s_0,1]\times L^{q^\prime}(\Omega).$ We want to prove that
		$u_k=\RR_p(s_k,h_k)$ has a strongly convergent subsequence
		in $L^{q}(\Omega).$
	
		For all $k\in\N,$ $u_k$ satisfies
		\[
			|u_k|_{W^{s_k,p}(\R^n)}^p
			=\int_{\Omega}h_k(x)u_k(x)\, dx.	
		\]
		Then, by H\"older inequality and using $q< p_{s_0}^\star,$ 
	 	we have
		\begin{equation}
			\label{eq:al1}
			|u_k|_{W^{s_k,p}(\R^n)}^p\le\|h_k\|_{L^{q^\prime}(\Omega)}
		\|u_k\|_{L^q(\Omega)}\le C\|u_k\|_{W^{s_0,p}(\Omega)}
	\end{equation}
	where $C$ is a constant independent of $k.$ 
	Thus, by Lemma \ref{lem:inclu}, Lemma \ref{lem:poincare}
	and \eqref{eq:al1}, we get
	\[
		\|u_k\|_{W^{s_0,p}(\Omega)}\le C
	\]
	for some constant $C$ independent of $k.$ Hence 
	$\{u_k\}_{k\in\N}$  has a strongly convergent subsequence
	in $L^{q}(\Omega)$ due to $\{u_k\}_{k\in\N}$ is 
	bounded in $W^{s_0,p}(\Omega)$ and $1<q<p_{s_0}^\star.$
	
	\medskip
	
	Finally, we show that $\RR_{p}$ is continuous.
	
	\medskip
	
	Let $(s_k,h_k)\to(s,h)$ in $[s_0,1]\times L^{q^\prime}(\Omega)$
	as $k\to\infty,$ $u_k=\RR_p(s_k,h_k)$ $k\in\N,$ and 
	$u=\RR_p(s,h).$ We want to
	show that $u_k\to u$ strongly in $L^{q}(\Omega).$ In fact,
	we only need to show that $u$ is the only accumulation point of 
	$\{u_k\}_{k\in\N}$ due to $\RR_p$ is compact.
	
	Let $\{u_j\}_{j\in\N}$ be a subsequence of $\{u_k\}_{k\in\N}$
	converging to $v$ in $L^q(\Omega).$ We have to prove that $v=u.$
	
	Give $w\in \twsp$ we define
	\[
		|w|_{s,p}^p=
		         \begin{cases}
	                  |w|_{\wup}^p &\text{ if } s=1,\\
	             	  \K(1-s)|w|_{\wspr}^p &\text{ if } s\in(0,1).
	             \end{cases}
	\]
	
	Let $\tilde{v}$ be the continuation of $v$ by zero outside 
	$\Omega.$ Then, it is enough to prove that 
	\begin{equation}
	\label{eq:al2}
		\dfrac1p|\tilde{v}|_{s,p}^p-\int_\Omega v(x)h(x)\,dx\le
		\dfrac1p|w|_{s,p}^p-\int_\Omega w(x)h(x)\,dx
		\quad\forall w\in\twsp.
	\end{equation}
	
	On the other hand, we know that
	\begin{equation}
	\label{eq:al3}
		\dfrac1p|u_j|_{s_j,p}^p-\int_\Omega u_j(x)h_j(x)\,dx\le
		\dfrac1p|w|_{s_j,p}^p-\int_\Omega w(x)h_j(x)\,dx
	\end{equation}
	for all $w\in \widetilde{W}^{s_j,p}(\Omega).$
	
	\medskip
	
	Now we need consider the following two cases.
	
	\medskip
	
	{\it Case $s\neq 1.$} Since $u_j\to v$ strongly in 
	$L^{q}(\Omega),$ 
	we have that $u_j\to \tilde{v}$ a.e. in $\R^n.$ 
	Then, using that $h_j\to h$ strongly in $L^{q^\prime}(\Omega)$
	and by Fatou's lemma, we have
	\begin{equation}\label{eq:al4}
		\dfrac1p|\tilde{v}|_{s,p}^p-
		\int_\Omega v(x)h(x)\,dx
		\le\liminf_{j\to\infty}
		\dfrac1p|u_j|_{s_j,p}^p-\int_\Omega u_j(x)h_j(x)\,dx.
	\end{equation}
	Thus, for any $\phi\in C^\infty_0(\Omega),$ 
	by \eqref{eq:al4},
	\eqref{eq:al3} and dominate convergence theorem, we get
	\[
		\dfrac1p|\tilde{v}|_{s,p}^p-\int_\Omega v(x)h(x)\,dx
		\le \dfrac1p|\phi|_{s,p}^p-\int_\Omega 
		\phi(x)h(x)\,dx.
	\]
	Therefore, $v\in\twsp$ and 
	by density, \eqref{eq:al2} holds.
	
	\medskip
	
	{\it Case $s=1.$} Let $\phi\in C_0^\infty(\Omega).$
	By \eqref{eq:al3} and Remark \ref{re:sopcomp}, we have 
	\begin{align*}
         \limsup_{j\to\infty}\dfrac1p|u_j|_{s_j,p}^p-
		\int_\Omega v(x)h(x)\,dx
		&=\limsup_{j\to\infty}\dfrac1p|u_j|_{s_j,p}^p-
		\int_\Omega u_j(x)h_j(x)\,dx\\
		& \le \dfrac1p|\phi|_{1,p}^p
		-\int_\Omega \phi(x)h(x)\,dx,	
	\end{align*}
	due to $u_j\to v$ strongly in $L^q(\Omega)$ and 
	$h_j\to h$ strongly in $L^{q^\prime}(\Omega).$ Then
	\begin{equation}
	\label{eq:al5}
		\limsup_{j\to\infty}\dfrac1p|u_j|_{s_j,p}^p
		 \le  \dfrac1p|\phi|_{1,p}^p
		-\int_\Omega \phi(x)h(x)\,dx + \int_\Omega v(x)h(x)\,dx.
    \end{equation}
	Therefore
	\[
		|u_j|_{s_j,p}\le C
	\]
	for some constant $C$ independent of $j.$
	
	Thus, by Theorem \ref{teo:bbm2}, 
	there exist $w\in\wupc$ and a subsequence of 
	$\{u_{j}\}_{j\in\N},$ still denoted by $\{u_{j}\}_{j\in\N},$
	 such that
	\begin{align*}
		u_{j}\to w &\quad\mbox{strongly in } \lp\\
		u_{j}\rightharpoonup w&\quad\mbox{weakly in }
		 W^{1-\varepsilon,p}(\Omega)
	\end{align*}
	for all $\varepsilon> 0.$  Then $v=w,$ and $v\in\wupc.$
	
	On the other hand, given $\varepsilon>0,$  
	there exists $j_0\in\N$ such that $1-\varepsilon<s_j$ for all
	$j\ge j_0$ due to $s_j\to 1.$ Then,
	by Remark \ref{re:seminor(1-s)} 
	\begin{equation}
	\label{eq:al6}
		\K\varepsilon|u_j|_{W^{1-\varepsilon,p}(\Omega)}^p\le 
		2^{\varepsilon p}
		\mbox{diam}(\Omega)^{(s_j-1+\varepsilon)p}
		|u_j|_{s_j,p}^p \quad\forall j\ge 
		j_0.
	\end{equation}
	Thus, using $u_{j}\rightharpoonup v$ weakly in 
	$W^{1-\varepsilon,p}(\Omega)$ and  by \eqref{eq:al6} 
	and \eqref{eq:al5},
	\begin{align*}
		\dfrac{2^{-\varepsilon p}
		\mbox{diam}(\Omega)^{-\varepsilon p}}{p}&\K\varepsilon
		|v|_{W^{1-\varepsilon,p}(\Omega)}^p
		\le \liminf_{j\to\infty}\frac1p|u_j|_{s_j,p}^p\\
		&\le \dfrac1p|\phi|_{1,p}^p
		-\int_\Omega \phi(x)h(x)\,dx + \int_\Omega v(x)h(x)\,dx.
	\end{align*}
	\[
	\]
	Now, by Theorem \ref{teo:bbm1}, letting $\varepsilon\to0^+$ we 
	get
	\[
		\dfrac1{p}|v|_{W^{1,p}(\Omega)}^p
		\le \dfrac1p|\phi|_{1,p}^p
		-\int_\Omega \phi(x)h(x)\,dx + \int_\Omega v(x)h(x)\,dx.
	\]
	Thus, since $\phi$ is arbitrary, we have that
	\[
		\dfrac1{p}|v|_{W^{1,p}(\Omega)}^p- \int_\Omega v(x)h(x)\,dx
		\le \dfrac1p|\phi|_{1,p}^p
		-\int_\Omega \phi(x)h(x)\,dx\quad\forall\phi\in
		C_0^\infty(\Omega).
	\]
	Hence, by density, \eqref{eq:al2} holds. This completes the proof.
\qed \end{proof}

\begin{re}
	Let  $s_0\in(0,1),$ and $p\in(1,\infty).$ Then the operator
	\begin{align*}
		\RR_{p}\colon [s_0,1]\times L^{p^\prime}(\Omega)
		&\to \lp\\
		(s,h)&\to \RR_{s,p}(h)
	\end{align*}
	is completely continuous.	
\end{re}

\section{The eigenvalue problems with weight}\label{ffl}

In this section we show some results concerning the
the following eigenvalue problems
\begin{equation}
	\label{eq:ep}
	\begin{cases}
		\LL_{s,p}(u)=\lambda h(x)|u|^{p-2}u 
		&\text{ in } \Omega,\\
		u=0 &\text{ in } \R^n\setminus\Omega.
	\end{cases}	
\end{equation}

Here $\Omega$ is a bounded domain in $\R^n$ with Lipschitz boundary, 
$s\in(0,1],$ $p\in(1,\infty)$ and $h\in \mathcal{A}=\{f\in L^{\infty}(\Omega)\colon |\{x\in\Omega\colon f(x)>0\}|>0\}.$

\subsection{The case $s=1,$ the first $p-$eigenvalue} 
Let $\Omega\subset\R^n$ be a bounded domain with Lipschitz   
boundary, $p\in(1,\infty)$ and $h\in\mathcal{A}.$
%

The first eigenvalue $\lambda_1(1,p,h)$ can be characterized as
\[
	\lambda_1(1,p,h)\coloneqq\inf\left\{
	|u|_{\wup}^p
	\colon u\in\wupc, \int_{\Omega}h(x)|u(x)|^p \, dx
	=1\right\},
\]
and it is simple and isolated, see \cite{Anane}. For simplicity,
we omit mention of $h$ when $h\equiv 1,$ and thus we write 
$\lambda_1(1,p)$ in place of $\lambda_1(1,p,1).$

\subsection{Case $s\in(0,1),$ the first fractional $p-$eigenvalue}
Let $\Omega\subset\R^n$ be a bounded domain with Lipschitz   
boundary, $s\in(0,1),$ $p\in(1,\infty)$ and $h\in\mathcal{A}.$ In this section, we 
analyse  the (non-linear non-local) eigenvalue problems
\begin{equation}
	\label{eq:epf}
	\begin{cases}
		(-\Delta)^s_p u=\lambda h(x) |u|^{p-2}u &\text{ in } \Omega,\\
	 	u=0 &\text{ in } \mathbb{R}^n\setminus \Omega.
	\end{cases}
\end{equation}  

\medskip

A function $u\in\twsp$ 
is a weak solution of \eqref{eq:epf} if it
satisfies
\begin{equation}\label{eq:debil}
	\K(1-s)\mathcal{H}_{s,p}
	(u,v)=\lambda\int_\Omega h(x)|u(x)|^{p-2}u(x)v(x)
	\, dx\quad \forall v\in\twsp.
\end{equation}

\medskip

We say that $\lambda\in\mathbb{R}$ is a fractional $p-$eigenvalue
provided there exists a non-trivial weak solution 
$u\in\twsp$ of \eqref{eq:epf}. The function $u$ is a corresponding 
eigenfunction.

\medskip

The first fractional p-eigenvalue is
\begin{equation}
	\label{eq:RQ}
	\begin{aligned}
	\lambda_1(s,p,&h)\coloneqq\\
	&\K(1-s)\inf\!\left\{
	|u|_{\wspr}^p\colon u\in\twsp, \int_{\Omega}
	h(x)|u(x)|^pdx=1\right\}.
\end{aligned}
\end{equation}
As before, in the case $h\equiv1,$ for simplicity, 
we write $\lambda_1(s,p)$ in place of $\lambda_1(s,p,1).$

\medskip

First we want to mention that $\{u\in\twsp\colon\int_{\Omega}
h(x)|u(x)|^p\, dx=1\}\neq\emptyset$ due to $|\{x\in\Omega\colon h(x) >0\}
|>0.$ Therefore
$\lambda_1(s,p,h)$ is well defined and is non-negative.

\medskip

We also know that $\lambda_1(s,p)>0$ and there exists a non-negative 
function $u\in\twsp$ such that
\begin{itemize}
	\item $u>0$ in $\Omega,$ and $u=0$ in $\R^n\setminus\Omega;$
	\item $u$ is a minimizer of \eqref{eq:RQ} with $h\equiv1$;
	\item $u$ is a weak solution of \eqref{eq:epf} with 
		$\lambda=\lambda_1(s,p)$ and $h\equiv1,$ 
		that is $u$ is an eigenfunction of \eqref{eq:op} with eigenvalue 
		$\lambda_1(s,p).$
\end{itemize}
Moreover $\lambda_1(s,p)$ is simple, and 
if $sp>n$ then $\lambda_1(s,p)$  is isolated.
See \cite[Theorem 5, Theorem 14 and Theorem 19]{LL},
\cite[Theorems A1]{BF} and \cite[Theorem 4.2]{FP}.
	 
\medskip

The rest of this section is devoted to generalize these results 
 for the first eigenvalue of  \eqref{eq:epf} with $h\equiv1$.

\begin{teo}\label{teo:LL}
     Let $\Omega\subset\R^n$ be a bounded domain with Lipschitz 
     boundary, $s\in(0,1),$ $p\in(1,\infty),$ and $h\in\mathcal{A}.$ 
     There exists a non-negative function 
     $u\in\twsp,$ such that
     \begin{itemize}
		\item $u\neq0$ in $\Omega;$
		\item $u$ is a minimizer of \eqref{eq:RQ};
		\item $u$ is a weak solution of \eqref{eq:epf} with 
				$\lambda=\lambda_1(s,p,h),$ that is $u$ is an 
				eigenfunction of \eqref{eq:op} with eigenvalue 
				$\lambda_1(s,p,h).$
     \end{itemize}
\end{teo}
\begin{proof}
	Let $\{u_j\}_{j\in\mathbb{N}}$ be a minimizing sequence, that is
	$u_j\in\twsp,$ 
	\[
		\int_{\Omega} h(x) |u_j(x)|^p \, dx=1 \mbox{ and }
		\lim_{j\to\infty}\K(1-s)|u_j|_{W^{s,p}(\R^n)}^p=\lambda(s,p,h).
	\]
	Then $\{u_j\}_{j\in\mathbb{N}}$ is bounded in $\twsp.$ Therefore, 
	there exit a subsequence, still denoted by 
	$\{u_j\}_{j\in\mathbb{N}},$ and $u\in\twsp$ such that
	\begin{align*}
		u_j\rightharpoonup u&\mbox{ weakly in }\twsp,\\
		u_j\to u &\mbox{ strongly in }\lp.
	\end{align*}
	Thus 
	\[
		\int_{\Omega}h(x)|u(x)|^p\, dx = 1
	\] 
	and
	\[
		\K(1-s)|u|_{\wspr}^p\le 
		\lim_{j\to\infty}\K(1-s)|u_j|_{\wspr}^p=\lambda(s,p,h).
	\]
	Then   $\K(1-s)|u|_{\wspr}^p=\lambda(s,p,h),$ that is 
	$u$ is a minimizer of \eqref{eq:RQ}. It is easy to see that $|u|$ is 
	also a minimizer of \eqref{eq:RQ}, this shows that there exists 
	a non-negative minimizer of \eqref{eq:RQ}.
	
	Finally, by the Lagrange multiplier rule 
	(see \cite[Theorem 2.2.10]{PK}) there exist 
	$a,b\in\R$ such that $a+b\neq0,$ and
	\[
		a\K(1-s)\mathcal{H}_{s,p}(u,v)+
		b\int_{\Omega}h(x)|u(x)|^{p-2}u(x)v(x)\, dx=0 \quad\forall 
		v\in\twsp.
	\]
	If $a=0,$ then $b\neq0$ and taking $v=u,$ we get 
	$\int_{\Omega}h(x)|u(x)|^p\, dx=0$
	a contradiction because $\int_{\Omega}h(x)|u(x)|^p\, dx=1.$ Hence 
	$a\neq0,$ and without any loss of generality, we can assume that 
	$a=1.$ Then
	\[
		\K(1-s)\mathcal{H}_{s,p}(u,v)+
		b\int_{\Omega}h(x)|u(x)|^{p-2}u(x)v(x)\, dx=0 \quad\forall 
		v\in\twsp.
	\]
	Again, taking $v=u$ and using that 
	\[
		\K(1-s)\mathcal{H}_{s,p}(u,v)=\K(1-s)|u|_{\wspr}^p=\lambda_1(s,p,h)
	\]
	and 
	\[
		\int_{\Omega}h(x)|u(x)|^p\, dx=1,
	\] 
	we have $b=-\lambda_1(s,p,h).$
\qed \end{proof}

Our next aim is to show that a non-negative eigenfunction associated to
$\lambda_1(s,p,h)$ is in really positive. For this we will need a strong 
minimum principle.

\medskip

We start by a definition. Let $p\in(1,\infty),$ $s\in(0,1),$  
$h\in \mathcal{A},$ and $\lambda\in\R.$ 
We say that $u\in \twsp$ is a weak 
super-solution of \eqref{eq:epf} 
if
\[
	\K(1-s)\mathcal{H}_{s,p}(u,v)\ge
	\lambda\int_{\Omega}h(x)|u(x)|^{p-1}u(x)	v(x)\, dx
\]
for all $v\in\twsp,v\ge0.$

\medskip

Following the proof of the Di Castro-Kussi-Palatucci logarithmic lemma (see 
\cite[Lemma 1.3]{di2014local}) we have the following result.

\begin{lem}\label{lem:DKP}
	Let $\Omega$ be a bounded domain, 
	$s\in(0,1), p\in(1,\infty),$ $h\in\mathcal{A},$ $\lambda>0$
	and $u$ be a weak super-solution of  \eqref{eq:epf} 
	such that $u\geq 0$ in $B_R(x_0)\subset\subset\Omega.$
	Then for any $B_r=B_r(x_0)\subset B_{\nicefrac{R}2}(x_0)$ and 
	$0<\delta<1$
	\begin{align*}
		\int_{B_r}\!\!\int_{B_r}&\left|
		\log\left(\dfrac{u(x)+\delta}{u(y)+\delta}\right)\right|^p
		\dfrac{dxdy}{|x-y|^{n+sp}}\le\\
		&\le Cr^{n-sp}
		\left\{\delta^{1-p}r^{sp}\int_{\R^n\setminus B_{2r}}
		\dfrac{u^-(y)^{p-1}}{|y-x_0|^{n+sp}}\, dy+1
		\right\}+C\lambda\|h\|_{L^1(B_{2r})},
	\end{align*}
	where $u^-=\max\{-u,0\}$ and $C$ depends only on $n,s,$ 
	and $p.$
\end{lem}

\begin{proof}
	Let $0<r<\nicefrac{R}{2},$ $0<\delta$ and 
	$\phi\in C_0^\infty( B_{\nicefrac{3r}2})$ be such that 	
	\[
  		0\le \phi \le 1, \quad \phi\equiv1 \text{ in } B_r \quad \text
  		{ and } \quad|D\phi|<Cr^{-1} 
  		\text{ in } B_{\nicefrac{3r}2}\subset B_{R}. 
	\]
	Taking $v=(u+\delta)^{1-p}\phi^p$ as test function in 
	\eqref{eq:debil} we have that
	 \[
  		  	\lambda
  		  	\int_{B_{\nicefrac{3r}2}} h(x) \frac{u(x)^{p-1}}{(u(x)+\delta)^{p-1}} 
  		  	\phi(x)^p\, dx \le \K(1-s)
  		  	\mathcal{H}_{s,p}(u,(u+\delta)^{1-p}\phi^p). 
  	\]
	Then,using that $0\le u^{p-1}(u+\delta)^{1-p}\phi^p\le1$ in 
	$B_{\nicefrac{3r}{2}},$ 
	\begin{equation}\label{eq:DKP}
 		 \begin{aligned}
  		  	0\le  \K(1-s)
  		  	\mathcal{H}_{s,p}(u,(u+\delta)^{1-p}\phi^p)
  		  	+\lambda\|h\|_{L^1(B_{2r})}. 
  		\end{aligned}
	\end{equation}
	In the proof of Lemma 1.3 in \cite{di2014local}, it is shown that
	\begin{align*}
	  \mathcal{H}_{s,p}(u,(u+\delta)^{1-p}\phi^p) \le 
	  &\ Cr^{n-sp} \left\{\delta^{1-p}r^{sp}\int_{\R^n\setminus B_{2r}} 	
	  \dfrac{u{-}(y)^{p-1}}{|y-x_0|^{n+sp}}\, d y+1 \right\}\\
 		& -\int_{B_r} \int_{B_r} \left| 
 		\log\left(\dfrac{u(x)+\delta}{u(y)+\delta}\right)\right|^p 
 		 \dfrac{dxdy}{|x-y|^{n+sp}},
	\end{align*}
	where $C$ depends only on $n,s,$ and $p.$ Then, by \eqref{eq:DKP},
	the lemma holds.\qed
	\end{proof}

	To prove of the next theorem, we adapt the proof of
	Theorem A.1 in \cite{BF}.
	
\begin{teo}\label{teo:stm} 
	Let $\Omega$ be a bounded domain, 
	$s\in(0,1), p\in(1,\infty),$ $h\in \mathcal{A},$  $\lambda>0,$
	and $u$ be a weak super-solution of  \eqref{eq:epf} 
	such that $u\geq 0$ in $\Omega.$ If $u\neq0$ in $\Omega$ then
	$u>0$ a.e. in $\Omega.$
\end{teo}

\begin{proof}
	We start proving that if $K\subset\subset\Omega$ is a compact connected 
	set such that $u\not\equiv 0$ then $u>0$ a.e. in $K.$
	
	Since $K\subset\subset \Omega$ and $K$ is compact then there 
	exist $r\in(0,1)$ and $x_1,\dots,x_k \in K$ such that
	\[
		K\subset\{x\in\Omega\colon \dist(x,\partial\Omega)>2r\},
		\quad K\subset \bigcup_{j=1}^k B_{\nicefrac{r}2}(x_i),
		\quad
		\mbox{ and }\quad 
	\]
	\begin{equation}\label{eq:totalmenteauxiliar}
			|B_{\nicefrac{r}2}(x_i)\cap B_{\nicefrac{r}2}(x_i+1)|>0
		\quad \forall i\in\{1,\dots,k-1\}.
	\end{equation}
	Suppose, to the contrary, $|\{x\in K\colon u(x)=0\}|>0.$ Then
	there exists $i\in\{1,\dots,k\}$ such that 
	$Z=|\{x\in K\colon u(x)=0\}\cap B_{\nicefrac{r}2}(x_i)|$ has 
	positive measure.
	
	Given $\delta>0,$ in the proof of Theorem A.1 in \cite{BF},
	it is shown that
	\begin{align*}
		\int_{B_{\nicefrac{r}2}(x_i)}&
		\left|\log\left(1+\dfrac{u(x)}\delta\right)\right|^p \, dx\le\\
		&\le \dfrac{r^{n+sp}}{|Z|}\int_{B_{\nicefrac{r}2}(x_i)}
		\int_{B_{\nicefrac{r}2}(x_i)}\left|
		\log\left(\dfrac{u(x)+\delta}{u(y)+\delta}\right)\right|^p
		\dfrac{dxdy}{|x-y|^{n+sp}}.
	\end{align*}
	Then, by Lemma \ref{lem:DKP},
	\[
		\int_{B_{\nicefrac{r}2}(x_i)}
		\left|\log\left(1+\dfrac{u(x)}\delta\right)\right|^p \, dx\le
		\dfrac{C}{|Z|}\max\{r^{2n},r^{n+sp}\}
	\]
	with $C$ independent of $\delta.$ Then, passing to the limit as 
	$\delta$ goes to 0, we have that $u\equiv 0$ in $B_{\nicefrac{r}2}(x_i).$
	Thus, proceeding as in  the proof of Theorem A1 in
	\cite{BF} we can conclude that $u\equiv0$ in $K,$	that is
	a contradiction.
	
	Proceeding as in  the proof of Theorem A1 in
	\cite{BF} we can conclude the general case.
	\qed
\end{proof}

\begin{co}
	\label{co:bf1}
	Let $\Omega\subset\R^n$ be a bounded domain with Lipschitz 
     boundary. Let $s\in(0,1),$ $p\in(1,\infty),$ $h\in \mathcal{A},$
     and $u\in\twsp$ be a non-negative eigenfunction corresponding  to 
      $\lambda_1(s,p,h).$ Then $u>0$ almost everywhere in $\Omega.$
\end{co}

Observe that, if $u$ is an eigenfunction corresponding  to 
$\lambda_1(s,p,h),$ then either $hu_+\not\equiv0$ or $hu_-\not\equiv0,$
where $u_+=\max\{u,0\}$ and $u_-=\max\{u,0\}.$ 

On the other hand, for any function $v\colon \R^N\to\R$
\begin{equation}
  \label{eq:u+u-}
  \begin{aligned}
     |v_+(x)-v_+(y)|^p&\le |v(x)-v(y)|^{p-2}(v(x)-v(y))(v_+(x)-v_+(y)),\\
    |v_-(x)-v_-(y)|^p&\le -|v(x)-v(y)|^{p-2}(v(x)-v(y))(v_-(x)-v_-(y)),
  \end{aligned}
\end{equation}
for all $x,y\in\R^n.$ Therefore, if $hu_+\not\equiv0$  then $u_+$ is an eigenfunction corresponding  to 
$\lambda_1(s,p,h).$ Moreover, by Corollary \ref{co:bf1}, $u_+>0$ almost everywhere in $\Omega.$

We similarly deduce that if $hu_-\not\equiv0$ then $u_->0$ almost everywhere in $\Omega.$
Then the next result is proved.

\begin{co}
	\label{co:bf2}
      Let $\Omega\subset\R^n$ be a bounded domain with Lipschitz 
     boundary, $s\in(0,1),$ $p\in(1,\infty),$ and $h\in \mathcal{A}.$
     If $u\in\twsp$ is an eigenfunction corresponding  to 
      $\lambda_1(s,p,h),$ then either $u>0$ or $u<0$ almost everywhere in $\Omega.$
\end{co}

The proof of  the result given below follows from a careful reading  of
\cite[proof of Theorem 3.2]{FP}.

\begin{teo}\label{teo:linf}
	Let $\Omega$ be a bounded domain with Lipschitz boundary, 
	$s\in(0,1), p\in(1,\infty),$ $h\in \mathcal{A},$  $\lambda>0,$
	and $u\in\twsp$ be a weak solution to
	\eqref{eq:epf}. Then $u\in L^\infty(\R^n).$
\end{teo}

Now, we prove that $\lambda_1(s,p,h)$ is also simple when $h\not\equiv1.$
For this we need the following lemma. For the proof see Lemma 6.2  in 
\cite{Amghibech}.

\begin{lem}\label{lem:laux1}  
	Let $p\in(1,\infty).$
	For $v>0$ and $u\ge0,$ we have
	\[
		L(u,v)\ge0 \quad \mbox{in } \R^n\times\R^n
	\]
	where
	\[
			L(u,v)(x,y)=
		   |u(y)-u(x)|^p-|v(y)-v(x)|^{p-2}(v(y)-v(x))
		   \left(\dfrac{u(y)^p}{v(y)^{p-1}}-
		   \dfrac{u(x)^p}{v(x)^{p-1}}\right). 
	\]
	The equality holds a.e in $\R^n\times \R^n$
	if and only if $u=kv$ a.e. in $\R^n$ for some constant $k.$
\end{lem}

\begin{teo}\label{teo:autoval1}
	 Let $\Omega\subset\R^n$ be a bounded domain with Lipschitz 
     boundary, $s\in(0,1),$ $p\in(1,\infty),$ $h\in\mathcal{A},$ 
     and $u$ be a 
     positive eigenfunction corresponding to $\lambda_{1}(s,p,h).$
     If $\lambda>0$ is such that there exists 
     a non-negative eigenfunction $v$ 
     of \eqref{eq:op} with
     eigenvalue $\lambda,$ then $\lambda=\lambda_1(s,p,h)$ and
     there exists $k\in\R$ such that $v = ku$ a.e. in $\Omega.$  
\end{teo}
\begin{proof} 
	Since $\lambda_1(s,p,h)$ is the first eigenvalue we have
	that $\lambda_1(s,p,h)\le\lambda$. 
	Let $m\in\N$ and $v_m\coloneqq v+\dfrac1{m}.$

	We begin by proving that 
	$w_{m}\coloneqq \dfrac{u^{p}}{v_m^{p-1}}\in\twsp.$ 
	It is immediate that  $w_m=0$ in $\R^n\setminus\Omega$ and 
	$w_{m}\in L^{p}(\Omega),$ due to $u\in L^{\infty}(\Omega),$ see
	Theorem \ref{teo:linf}.
 
	On the other hand
	\begin{align*}
		|w_{m}(y)-w_{m}(x)|=&\left|
		\dfrac{u(y)^{p}-u(x)^p}{v_m(y)^{p-1}}
		+
		\dfrac{u(x)^p\left(v_m(x)^{p-1}-v_m(y)^{p-1}\right)}
		{v_m(y)^{p-1}v_m(x)^{p-1}}\right|\\
		\le& m^{p-1}\left|u(y)^{p}-u(x)^p\right|
		+\|u\|_{L^{\infty}(\Omega)}^p
		\dfrac{\left|v_m(y)^{p-1}-v_m(x)^{p-1}\right|}
		{v_m(y)^{p-1}v_m(x)^{p-1}}\\
		\le& m^{p-1}p(u(y)^{p-1}+u(x)^{p-1})|u(y)-u(x)|\\
		&+\|u\|_{L^{\infty}(\Omega)}^p(p-1)
		\dfrac{|v_m(y)^{p-2}+v_m(x)^{p-2}|}
		{v_m(y)^{p-1}v_m(x)^{p-1}}|v_m(y)-v_m(x)|\\
		\le& 2\|u\|_{L^{\infty}(\Omega)}^{p-1}m^{p-1}p|u(y)-u(x)|\\
		&+\|u\|_{L^{\infty}(\Omega)}^p(p-1)m^{p-1}
		\left(\dfrac1{v_m(y)}+\dfrac1{v_m(x)}\right)|v(y)-v(x)|\\
	 	\le& C(m,p,\|u\|_{L^{\infty}(\Omega)})
		\left(|u(y)-u(x)|+|v(y)-v(x)|\right)
	\end{align*} 
	for all $(x,y)\in\R^n\times\R^n.$ Hence 
	$w_{m}\in\twsp$ for all $m\in\N$ due to $u,v\in\twsp.$

	Then, by Lemma \ref{lem:laux1} and since 
	$u,v\in\twsp$ are two positive eigenfunctions of 
	problem \eqref{eq:op} with eigenvalue $\lambda_1(s,p,h)$ and 
	$\lambda$ respectively, we have
	\begin{align*}
    	0\le&
    	\int_{\R^n}\int_{\R^n} 
    	\dfrac{L(u,v_m)(x,y)}{|x-y|^{n+sp}} dxdy\\
    	\le& \int_{\R^n}\int_{\R^n}
    	\dfrac{|u(y)-u(x)|^p}{|x-y|^{n+sp}} dx dy\\ 
    	&-\int_{\R^n}\int_{\R^n}
    	\dfrac{|v(y)-v(x)|^{p-2}(v(y)-v(x))}{|x-y|^{n+sp}}
		\left(\dfrac{u(y)^p}{v_m(y)^{p-1}}-
   		\dfrac{u(x)^p}{v_m(x)^{p-1}}\right)dxdy\\
		\le&\dfrac{1}{\K(1-s)}\left\{ 
		\lambda_{1}(s,p,h)\int_{\Omega}h(x)u(x)^p\,dx-
		\lambda\int_{\Omega}h(x)v(x)^{p-1}
		\dfrac{u(x)^p}{v_m(x)^{p-1}}\, dx\right\}.
	\end{align*}
	By Fatou's lemma and the dominated convergence theorem
	\[
    	\int_{\R^n}\int_{\R^n} \dfrac{L(u,v)(x,y)}{|x-y|^{n+sp}}
    	\, dxdy= 0
	\]
	due to $\lambda_1(s,p,h)\le \lambda.$
	Then $L(u,v)(x,y)=0$ a.e. in $\R^n\times \R^n.$ Hence, by Lemma 
	\ref{lem:laux1}, $u=kv$ for 
	some constant $k>0.$
\qed \end{proof}

By Corollary \ref{co:bf2} and Theorem \ref{teo:autoval1}, we have that
$\lambda_1(s,p,h)$ is simple.

\begin{teo}\label{teo:autovsimple}
	Let $\Omega\subset\R^n$ be a bounded domain with Lipschitz 
     boundary, $s\in(0,1),$ $p\in(1,\infty),$ and $h\in\mathcal{A}.$ 
     Then $\lambda_{1}(s,p)$ is simple.
\end{teo}

Now, we get a lower bound for the measure of the nodal sets.

\begin{lem}\label{lem:autoval3}
	Let $\Omega\subset\R^n$ be a bounded domain with Lipschitz 
     boundary. Let $s\in(0,1),$ 
	$p\in(1,\infty,),$ $s_0\in(0,\min\{\nicefrac{n}{p},s\}),$
	and $h\in\mathcal{A}.$   
	If $u\in\twsp$ is an eigenfunction of 
    \eqref{eq:op} with eigenvalue $\lambda>\lambda_1(s,p,h),$ then  
    there exists a constant $C=C(n,p)$ 
    such that 
    \[
		\left(C
		\dfrac{s_0 (n-s_0p)^{p-1}\mbox{diam}(\Omega)^{s_0p}}{
		\mbox{diam}(\Omega)^{sp}\lambda
			\|h\|_{L^\infty(\Omega)}s_0p+\K}
		\right)^{\frac{p_{s_0}^\star}{p_{s_0}^\star-p}}\le |\Omega^\pm|.
    \]
    Here $\Omega^+=\{x\in\Omega\colon u(x)>0\}$ and  
    $\Omega^-=\{x\in\Omega\colon u(x)<0\}.$ 
\end{lem}
\begin{proof} 
	 By Theorem \ref{teo:autoval1}, 
	$u$ changes sign then $u^+\not\equiv0.$ In addition, $u^{+}\in\twsp.$ 
	It follows from \eqref{eq:u+u-} that
	\begin{equation}\label{eq:aumedes1}
		\begin{aligned}
			\K(1-s)|u^+|_{\wspr}^p
		    &\le \K(1-s)\mathcal{H}_{s,p}(u,u^+)
		    \le \lambda\int_{\Omega^+}h(x)|u^+(x)|^p dx\\
		    &\le \lambda\|h\|_{L^\infty(\Omega)}
		    \int_{\Omega^+}|u^+(x)|^p dx.
		 \end{aligned}	   
	\end{equation}
    
    On the other hand, by Remark \ref{re:mazya}, there exists 
    a constant $C=C(n,p)$ such 	that
    \begin{align*}
			&\|u^+\|_{L^{p_{s_0}^\star}(\R^n)}^p
			\le\\
			& \dfrac{C}{(n-s_0p)^{p-1} \mbox{diam}(\Omega)^{s_0p}}
			\left(\mbox{diam}(\Omega)^{sp}(1-s)|u^+|_{W ^{s,p}(\R^n)}^p + 
			\dfrac1{s_0p}\|u^+\|^p_{\lp}\right).
		\end{align*}
	Then, by \eqref{eq:aumedes1} and H\"older's inequality, we get
	\begin{align*}
			&\|u^+\|_{L^{p_{s_0}^\star}(\R^n)}^p
			\le\\
			& \dfrac{C}{(n-s_0p)^{p-1} \mbox{diam}(\Omega)^{s_0p}}
			\left(
			\dfrac{\mbox{diam}(\Omega)^{sp}\lambda
			\|h\|_{L^\infty(\Omega)}}{\K}+ 
			\dfrac1{s_0p}\right)\|u^{+}\|_{L^{p_{s_0}^\star}(\Omega)}^{p}
		   |\Omega^+|^{\frac{p_{s_0}^\star-p}{p_{s_0}^\star}}.
		\end{align*}
    Hence
    \[
		\left(C
		\dfrac{\K s_0p (n-s_0p)^{p-1}\mbox{diam}(\Omega)^{s_0p}}{
		\mbox{diam}(\Omega)^{sp}\lambda
			\|h\|_{L^\infty(\Omega)}s_0p+\K}
		\right)^{\frac{p_{s_0}^\star}{p_{s_0}^\star-p}}\le |\Omega^+|.
    \]

	\medskip
	
	In order to prove the second inequality, it will suffice 
	to proceed as above, using the function $u^-(x)=\max\{0,-u(x)\}$ 
	instead of
	$u^+.$
\qed \end{proof}	

Finally, we show that the first eigenvalue is isolated, see also
\cite{2014arXiv1409.6284B}.

\begin{teo}
	\label{teo:isolated}
	Let $\Omega\subset\R^n$ be a bounded domain with Lipschitz 
    boundary, $s\in(0,1),$ $p\in(1,\infty),$ and $h\in\mathcal{A}.$ 
    Then the first eigenvalue is isolated.  	
\end{teo}

\begin{proof}
	By the definition of $\lambda_1(s,p,h),$ we have
	that $\lambda_1(s,p,h)$ is left--isolated.
	
	To prove that $\lambda_1(s,p,h)$ is right--isolated, we argue by 
	contradiction. We assume that there exists a  sequence of 
	eigenvalues $\{\lambda_k\}_{k\in\N}$ such that  
	$\lambda_k>\lambda_1(s,p,h)$ and 
	$\lambda_k\searrow \lambda_1(s,p,h)$ as $k\to \infty.$  
	Let $u_k$ be an eigenfunction associated to $\lambda_k,$ 
	we can assume that 
	\[
		\int_{\Omega}h(x)|u_k(x)|^p\,dx=1.
	\] 
	Then $\{u_k\}_{k\in\N}$
	is bounded in $\twsp$ and therefore
	we can extract a subsequence (that we still denoted by 
	$\{u_k\}_{k\in\N}$) 
	such that 
	\begin{align*}
		u_k\rightharpoonup u&\quad \mbox{ weakly in } \twsp,\\
		u_k\to u &\quad \mbox{ strongly in } L^p(\Omega).
	\end{align*}
	Then 
	\[
		\int_{\Omega}h(x)|u(x)|^p\,dx=1
	\]
	and
	\begin{align*}
		\K(1-s)|u|_{\wspr}^p&
		\le\K(1-s)\liminf_{k\to\infty}|u_k|_{\wspr}^p
		=\lim_{k\to\infty}\lambda_k\int_{\Omega}
					h(x)|u_k(x)|^p\,dx\\
					&=\lambda_1(s,p,h)\int_{\Omega}
					h(x)|u(x)|^p\,dx.
	\end{align*}
	Hence, $u$ is an eigenvalue of \eqref{eq:op} with eigenvalue 
	$\lambda_1(s,p,h).$ By Corollary \ref{co:bf2}, we can assume that
	$u>0.$
	
	On the other hand, by the Egorov's theorem, for any $\varepsilon>0$
	there exists a subset $A_\varepsilon$ of $\Omega$ such that 
	$|A_\varepsilon|<\varepsilon$ and $u_k\to u>0$ uniformly in 
	$\Omega\setminus A_{\varepsilon}.$ This contradicts the fact that,
	by Lemma \ref{lem:autoval3},
	\[
		\left(C
		\dfrac{s_0 (n-s_0p)^{p-1}\mbox{diam}(\Omega)^{s_0p}}{
		\mbox{diam}(\Omega)^{sp}\lambda_k
			\|h\|_{L^\infty(\Omega)}s_0p+\K}
		\right)^{\frac{p_{s_0}^\star}{p_{s_0}^\star-p}}\le 
		|\{x\in\Omega\colon u_k(x)<0\}|
    \]
	where $s_0\in(0,\min\{s,\nicefrac{n}{p}\})$ and $C$ depends on $n$
	and $p.$
	This proves the theorem.
\qed \end{proof}

\subsection{Global properties}
In the rest of this section, for simplicity, we will take $h\equiv1.$
\begin{lem} 
	\label{lem:eigcont}	
	 Let $\Omega\subset\R^n$ be a bounded domain with Lipschitz 
     boundary and $p\in(1,\infty).$
	 The first eigenvalue  function
	$\lambda_1(\cdot,p)\colon(0,1]\to\mathbb{R}$
	is continuous.
\end{lem}
\begin{proof}
	Let $\{s_j\}_{j\in\N}$ be a sequence in $(0,1]$ convergent
	to $s\in(0,1].$ We will show that
	\begin{equation}
	\label{eq:aucon1}
		\lim_{j\to\infty}\lambda_1(s_j,p)=\lambda_1(s,p).
	\end{equation}
	
	\medskip
	
	We need to consider two cases: $s\in(0,1)$ and $s=1.$
	
	\medskip
	
	{\it Case} $s\in(0,1).$ Let $\phi\in C_0^\infty(\Omega),$ 
	$\phi\not\equiv0.$ Then 
	\[
		\lambda_1(s_j,p)\le \K(1-s_j)
		\dfrac{\displaystyle\int_{\R^n}\int_{\R^n}
		\dfrac{|\phi(x)-\phi(y)|^p}{|x-y|^{n+s_jp}}\, dxdy}
		{\displaystyle\int_\Omega |\phi(x)|^p \, dx}
	\]
	for all $j\in\N.$ Therefore, by dominated convergence theorem,
	\[
		\limsup_{j\to\infty}\lambda_1(s_j,p)\le \K(1-s)
		\dfrac{\displaystyle\int_{\R^n}\int_{\R^n}
		\dfrac{|\phi(x)-\phi(y)|^p}{|x-y|^{n+sp}}\, dxdy}
		{\displaystyle\int_\Omega |\phi(x)|^p \, dx}.
	\]
	As $\phi$ is arbitrary 
	\[
		\limsup_{j\to\infty}\lambda_1(s_j,p)\le \lambda_1(s,p) 
	\]
	due to \eqref{eq:RQ}.
	
	Thus, to prove \eqref{eq:aucon1}, we need to show that
	\[
		\liminf_{j\to\infty}\lambda_1(s_j,p)\ge \lambda_1(s,p).
	\]
	Let $\{s_k\}_{k\in\N}$ be a subsequence of $\{s_j\}_{j\in\N}$
	such that
	\begin{equation}
	\label{eq:aucon2}
		\lim_{k\to\infty}\lambda_1(s_k,p)
		=\liminf_{j\to\infty}\lambda_1(s_j,p).
	\end{equation}
	
	Let $u_k$  be an eigenfunction of \eqref{eq:op} with eigenvalue 
	$\lambda_1(s_k,p)$ such that $\|u_k\|_{L^p(\Omega)}=1.$
	Since
	\begin{equation}\label{eq:nueva1}
		\begin{aligned}
			\lim_{k\to\infty}\K(1-s_k)|u_k|_{W^{s_k,p}(\R^n)}^p&=
			\lim_{k\to\infty}\lambda_1(s_k,p)
			=\liminf_{j\to\infty}\lambda_1(s_j,p)\\
			&\le\limsup_{j\to\infty}\lambda_1(s_j,p)
			\le\lambda_1(s,p).
		\end{aligned}	
	\end{equation}
	Then $\{\K(1-s_k)|u_k|_{W^{s_k,p}(\R^n)}^p\}_{k\in\N}$ is bounded,
	therefore $\{|u_k|_{W^{s_k,p}(\R^n)}^p\}_{k\in\N}$ is bounded.
	
	On the other hand, given $\varepsilon>0$ there exists 
	$k_0\in\N$ such
	that $s-\varepsilon<s_k$ for all $k\ge k_0$ 
	and, by Lemma 
	\ref{lem:inclu}, we have
	\begin{equation}
	\label{eq:aucon4}
			|u_k|_{W^{s-\varepsilon,p}(\R^n)}^p\le  
			|u_k|_{W^{s_k,p}(\R^n)}^p
			+C(n,p)\left(\dfrac{1}{(s-\varepsilon)p}
			-\dfrac{1}{s_k p}\right)
			\|u_k\|_{\lp}^p
	\end{equation}
	for all $k\ge k_0.$
	Thus, using that $\|u_k\|_{\lp}=1$ for all $k\in\N$
	and $\{|u_k|_{W^{s_k,p}(\R^n)}^p\}_{k\in\N}$ is bounded, 
	we have that $\{u_k\}_{k\ge k_0}$ 
	is bounded in $\widetilde{W}^{s-\varepsilon,p}(\Omega).$ 
	Then there exists $u\in \widetilde{W}^{s-\varepsilon,p}(\Omega)$ 
	such that, for a subsequence that 
	we still call $\{u_k\}_{k\ge k_0},$ 
	\begin{equation}
	\label{eq:aucon5}
		\begin{aligned}
		u_k\rightharpoonup u&\quad \mbox{ weakly in } 
		{W}^{s-\varepsilon,p}(\R^n),\\
		u_k\to u &\quad \mbox{ strongly in } L^p(\Omega).
		\end{aligned}
	\end{equation}
	Then $\|u\|_{\lp}=1$ and by \eqref{eq:aucon5}, \eqref{eq:aucon4} and 
	\eqref{eq:aucon2}, we get
	\begin{align*}
		&\K(1-s)|u|_{W^{s-\varepsilon,p}(\R^n)}^p
		\le \liminf_{k\to \infty}
		\K(1-s_k)|u_k|_{W^{s-\varepsilon,p}(\R^n)}^p\\
		&\le\liminf_{k\to \infty}\left\{ 
		\K(1-s_k)|u_k|_{W^{s_k,p}(\R^n)}^p
		\vphantom{+C(n,p)\K(1-s_k)
		\left(\dfrac{1}{(s-\varepsilon)p}-
		\dfrac{1}{s_k p}\right)
		\|u_k\|_{\lp}^p}
		+C(n,p)\K(1-s_k)
		\left(\dfrac{1}{(s-\varepsilon)p}-
		\dfrac{1}{s_k p}\right)\right\}\\
		&=\liminf_{k\to \infty}\left\{ 
		\lambda_1(s_k,p)
		\vphantom{+C(n,p)\K(1-s_k)
		\left(\dfrac{1}{(s-\varepsilon)p}-
		\dfrac{1}{s_k p}\right)}
		+C(n,p)\K(1-s_k)
		\left(\dfrac{1}{(s-\varepsilon)p}-
		\dfrac{1}{s_k p}\right) \right\}\\
		&=\liminf_{j\to \infty}
		\lambda_1(s_j,p)+C(n,p)\K(1-s)
		\left(\dfrac{1}{(s-\varepsilon)p}-
		\dfrac{1}{s p}\right).
	\end{align*}
	As $\varepsilon>0$ is arbitrary, by Fatou Lemma, we have
	\begin{equation}
	\label{eq:aucon6}
		\begin{aligned}
			\K(1-s)|u|_{\wspr}^p
			&\le \K(1-s)\liminf_{\varepsilon\to 0^+}
			|u|_{W^{s-\varepsilon,p}(\R^n)}^p \\
			&\le\liminf_{j\to \infty}
			\lambda_1(s_j,p).
		\end{aligned}
	\end{equation}
	
	Since $\|u\|_{L^p(\Omega)}=1,$ by \eqref{eq:RQ} and \eqref{eq:aucon6}, we get
	\[
		\lambda_1(s,p)\le \K(1-s)|u|_{\wspr}^p
		\le\liminf_{j\to \infty}\lambda_1(s_j,p).
	\]
	
	{\it Case} $s=1.$ Let $\phi\in C_0^\infty(\Omega),$
	$\phi\not\equiv0.$ Then, for any $j\in\N$ 
	\[
		\lambda_1(s_j,p)\le \K(1-s_j)
		\dfrac{\displaystyle\int_{\R^n}\int_{\R^n}
		\dfrac{|\phi(x)-\phi(y)|^p}{|x-y|^{n+s_jp}}\, dxdy}
		{\displaystyle\int_\Omega |\phi(x)|^p \, dx}.
	\]
	Thus, by Remark \ref{re:sopcomp} and the above inequality, we get
	\[
		\limsup_{j\to\infty}\lambda_1(s_j,p)\le
		\dfrac{\displaystyle\int_\Omega|\nabla\phi(x)|^p\, dx }
		{\displaystyle\int_\Omega |\phi(x)|^p \, dx}.
	\]
	As $\phi$ is arbitrary 
	\[
		\limsup_{j\to\infty}\lambda_1(s_j,p)\le \lambda_1(1,p). 
	\]
	
	As in the previous case, to prove \eqref{eq:aucon1}, 
	we need to show that
	\[
		\liminf_{j\to\infty}\lambda_1(s_j,p)\ge \lambda_1(s,p).
	\]
	Let $\{s_k\}_{k\in\N}$ be a subsequence of $\{s_j\}_{j\in\N}$
	such that
	\begin{equation}
	\label{eq:aucon7}
		\lim_{k\to\infty}\lambda_1(s_k,p)
		=\liminf_{j\to\infty}\lambda_1(s_j,p).
	\end{equation}
	Let $u_k$ be an eigenfunction of \eqref{eq:op} with eigenvalue 
	$\|u_k\|_{\lp}=1.$Since
	\begin{equation}\label{eq:nueva2}
		\begin{aligned}
			\lim_{k\to\infty}\K(1-s_k)|u_k|_{W^{s_k,p}(\R^n)}^p&=
			\lim_{k\to\infty}\lambda_1(s_k,p)
			=\liminf_{j\to\infty}\lambda_1(s_j,p)\\
			&\le\limsup_{j\to\infty}\lambda_1(s_j,p)
			\le\lambda_1(s,p).
		\end{aligned}	
	\end{equation}
	Then $\{\K(1-s_k)|u_k|_{W^{s_k,p}(\R^n)}^p\}_{k\in\N}$ is bounded,
	therefore, by Theorem \ref{teo:bbm2}, we can extract 
	a subsequence (that we still denote by $\{u_k\}_{k\in\N}$) such that
	\begin{align}
		\label{eq:aucon10}u_{k}\rightharpoonup u&\quad\mbox{weakly in }
		W^{1-\varepsilon,p}(\Omega),\\
		\label{eq:aucon9}u_{k}\to u &\quad\mbox{strongly in } \lp,
	\end{align}
	for all $\varepsilon> 0,$ to some $u\in \wupc.$ 
	Thus, by \eqref{eq:aucon9}, we have $\|u\|_{\lp}=1$ and then
	\begin{equation}
	\label{eq:aucon11}
		\lambda_1(1,p)\le\int_\Omega |\nabla u|^p\, dx.
	\end{equation}
	
	On the other hand, given $\varepsilon>0$ there exists $k_0$ such that
	$1-\varepsilon<s_k$ for all $k\ge k_0.$ Then, by Remark 
	\ref{re:seminor(1-s)}, we get
	\[
		\K\varepsilon|u_k|_{W^{1-\varepsilon,p}(\Omega)}^p\le 
		2^{\varepsilon p}\mbox{diam}(\Omega)^{s_k-1+\varepsilon}
		\K(1-s_k)|u_k|_{W^{s_k,p}(\R^n)}^p
		\quad \forall k\ge k_0.
	\]
	Thus, by \eqref{eq:aucon10},
	\begin{equation}
	\label{eq:aucon12}
	\begin{aligned}
		\K\varepsilon|u|_{W^{1-\varepsilon,p}(\Omega)}^p&\le 
		\liminf_{k\to\infty}
		\K\varepsilon|u_k|_{W^{1-\varepsilon,p}(\Omega)}^p\\
		&\le 
		2^{\varepsilon p}
		\mbox{diam}(\Omega)^{\varepsilon}\liminf_{j\to\infty}\lambda_1(s_j,p).
	\end{aligned}	
	\end{equation}

	As $\varepsilon>0$ is arbitrary, by \eqref{eq:aucon12} and 
	Theorem \ref{teo:bbm1}, we have that
	\begin{equation}
	\label{eq:aucon13}
		\int_\Omega|\nabla u(x)|^p dx=
		\lim_{\varepsilon\to 0^+}
		\K\varepsilon|u|_{W^{1-\varepsilon,p}(\Omega)}^p\le
		\liminf_{j\to\infty}\lambda_1(s_j,p).
	\end{equation}

	Finally, by \eqref{eq:aucon11} and \eqref{eq:aucon13}, we get
	\[
		\lambda_1(1,p)\le
		\liminf_{j\to\infty}\lambda_1(s_j,p).
	\]
	This completes the proof.
\qed \end{proof}

For the proof of the following lemma we borrow ideas from 
\cite[Lemma 2.3]{delPino}.

\begin{lem}\label{lem:lgi}
	For every interval $[a,b]\subset(0,1]$  there is a 
	$\delta>0$ such that for all $s\in[a,b]$ there is no 
	eigenvalue of \eqref{eq:ep} in 
	$(\lambda_1(s,p),\lambda_1(s,p)+\delta].$	
\end{lem}

\begin{proof}
	Suppose the lemma were false. Then we could find sequences
	$\{s_k\}_{k\in\N}$ in $(0,1],$ $\{\lambda_{k}\}_{k\in\N}$ in
	$\R_+$ and $\{u_k\}_{k\in\N}$ in $\twsp\setminus\{0\}$ such that
	\[
		\lim_{k\to\infty} s_k =s\in(0,1], \quad 
		\lambda_k>\lambda_1(s_k,p)\quad\forall k\in\N,\quad 
		\lim_{k\to\infty} \lambda_k-\lambda_1(s_k,p)=0,
	\]
	and for all $k\in\N$ $\|u_k\|_{\lp}=1 $ and
	\begin{equation}\label{eq:algi1}
		u_k=\RR_{s_k,p}(\lambda_k|u_k|^{p-2}u_k).
	\end{equation}
	
	By Lemma \ref{lem:eigcont}, 
	\begin{equation}\label{eq:algi2}
		\lim_{k\to\infty}\lambda_k=\lambda_1(s,p).
	\end{equation}
	
	\medskip
	
	On the other hand $\{|u_k|^{p-2}u_k\}_{k\in\N}$ is
	bounded in $L^{p^\prime}(\Omega)$ due to 
	$\|u_k\|_{\lp}=1$ for all $k\in\N.$ Then, by Lemma \ref{lem:contdeR},
	there exist $u\in\lp$ and a subsequence of $\{u_k\}_{k\in\N},$ still
	denoted $\{u_k\}_{k\in\N},$ such that $u_k\to u$ in $\lp$ as 
	$k\to\infty.$ Thus
	\begin{equation}\label{eq:algi3}
		|u_k|^{p-2}u_k\to |u|^{p-2}u \quad\text{strongly in }
		L^{p^\prime}(\Omega).
	\end{equation}
	
	Then, passing to the limit in \eqref{eq:algi1}, using
	\eqref{eq:algi2}, \eqref{eq:algi3} and Lemma  \ref{lem:contdeR},
	we get
	\[
		u=\RR_{s,p}(\lambda_1(s,p)|u|^{p-2}u).
	\]
	Therefore $u$ is an eigenfunction associated to $\lambda_1(s,p).$
	Then, by Corollary \ref{co:bf2} and Theorem \ref{teo:autoval1}, 
	we may assume without loss of generality,  that $u>0.$ 
	
	\medskip
	
	On the other hand, given $s_0\in(0,\min\{s,\nicefrac{n}p\})$
	there exists $k_0\in N$ such that $s_k\ge s$ for all $k\ge k_0$
	due to $s_k\to s$ as $k\to\infty.$ Thus, by Lemma \ref{lem:autoval3}, 
	we get
	\[
		\left(C
		\dfrac{s_0 (n-s_0p)^{p-1}\mbox{diam}(\Omega)^{s_0p}}{
		\mbox{diam}(\Omega)^{s_kp}\lambda_k
			\|h\|_{L^\infty(\Omega)}s_0p+\K}
		\right)^{\frac{p_{s_0}^\star}{p_{s_0}^\star-p}}\le 
		|\{x\in\Omega\colon u_k(x)<0\}|\quad \forall k\ge k_0.
    \]
	Then, since $u_k\to u$ in $\lp,$ $u$ must change its sign in $\Omega,$
	contrary to the fact that $u>0.$ 
\qed \end{proof}

\section{Bifurcation}\label{bifurcation}
Let $\Omega\subset\R^n$ be a bounded domain with Lipschitz boundary, 
$s\in(0,1],$ and 
$p\in(1,\infty).$ In this section we consider the following non-linear problem:

\begin{equation}
	\label{eq:elproblema}
	\begin{cases}
		\LL_{s,p}(u)=\lambda|u|^{p-2}u + f(x,u,\lambda) 
		&\text{ in } \Omega,\\
		u=0 &\text{ in } \R^n\setminus\Omega,
	\end{cases}	
\end{equation}
where $f\colon\Omega\times\R\times\R\to\R$
is a function such that
\begin{enumerate}
	\item $f$ satisfies a Carath\'eodory condition in the first two variables;
	\item $f(x,t,\lambda)=o(|t|^{p-1})$ near $t=0,$ uniformly a.e. with respect 
			to $x$ and uniformly with respect to $\lambda$ on bonded sets;
	\item There exists $q\in (1,p^\star_s)$ such that
			\[
				\lim_{|t|\to\infty}\dfrac{|f(x,t,\lambda)|}{|t|^{q-1}}=0
			\]
			uniformly a.e. with respect 
			to $x$ and uniformly with respect to $\lambda$ on bounded sets.
\end{enumerate}

\medskip

A pair $(\lambda,u)\in \R\times\twsp$ is a weak solution of 
\eqref{eq:elproblema} if 
\[
	\mathfrak{H}_{s,p}(u,v)=\int_\Omega 
		\left(\lambda|u(x)|^{p-2}u(x) + f(x,u,\lambda)\right)
		v(x)\, dx,
\]
for all $v\in\twsp.$ Here
\[
	\mathfrak{H}_{s,p}(u,v)=
	\begin{cases}
		\K(1-s)\mathcal{H}_{s,p}(u,v) &\text{ if } 0<s<1,\\
		\displaystyle\int_{\Omega}|\nabla u(x)|^{p-2}\nabla u(x)
		\nabla v(x)\, 
		dx
		&\text{ if } s=1.
	\end{cases}
  \]

\begin{re} 
	The pair $(\lambda,u)$ is weak solution of
	\eqref{eq:elproblema} iff $(\lambda,u)$ satisfies
	\[
	 	u=\mathscr{R}_{\lambda}(u)
	\]
	where $\mathscr{R}_{\lambda}(u)=\RR_{s,p}(\lambda|u|^{p-2}u+F(\lambda,u))$
	and $F(\lambda,\cdot)$ is the Nemitsky operator associated with
	$f.$
\end{re}

\medskip

We say that $(\lambda,0)\in\R\times\twsp$ is a bifurcation point of 
\eqref{eq:elproblema} if in any neighbourhood of $(\lambda,0)$ in 
$\R\times\twsp$ there exists a 
nontrivial solution of \eqref{eq:elproblema}. 

\medskip

The proof of the following result  is analogous to that of Proposition 
2.1 in \cite{delPino}

\begin{lem}\label{lem:bifaut} 
	Let $\Omega\subset\R^n$ be a bounded domain with Lipschitz boundary,
	$s\in(0,1],$ and $p\in(1,\infty).$ If $(\lambda,0)$ is a
	 bifurcation point of \eqref{eq:elproblema} then
	 $\lambda$ is an eigenvalue of  following
	eigenvalue problems
	\begin{equation}
	\label{eq:paut1}
	\begin{cases}
		(-\Delta)_p^su=\lambda|u|^{p-2}u 
		&\text{ in } \Omega,\\
		u=0 &\text{ in } \R^n\setminus\Omega.
	\end{cases}	
\end{equation}
\end{lem}

Let
\[
	\lambda_2(s,p)\coloneqq
	\inf\left\{\lambda>\lambda_1(s,p)\colon \lambda \mbox{ is an
	 eigenvalue of \eqref{eq:paut1}}\right\}.
\] 
For $\lambda<\lambda_1(s,p)$
or $\lambda_1(s,p)<\lambda<\lambda_2(s,p)$ the function $u\equiv0$
is the unique solution of
\[
	u=\RR_{s,p}(\lambda|u|^{p-2}u).
\]

Then for $\lambda<\lambda_1(s,p)$
or $\lambda_1(s,p)<\lambda<\lambda_2(s,p)$ we define the 
completely continuous operator
$\T_{s,p}^{\lambda}:\twsp\to\twsp$
\[
	\T_{s,p}^{\lambda}(u)\coloneqq\RR_{s,p}(\lambda|u|^{p-2}u).
\]
Thus
\[
	\deg_{\twsp}(I-\T_{s,p}^{\lambda},B(0,r),0)
\]
is well defined for any $\lambda<\lambda_1(s,p)$
or $\lambda_1(s,p)<\lambda<\lambda_2(s,p)$ and $r>0.$

\begin{teo}\label{teo:gradoT}
	Let $s\in(0,1],$  $p\in(1,\infty),$ and $r>0.$ Then
	\[
		\deg_{\widetilde{W}^{s,p}(\Omega)}
		(I-\T_{s,p}^{\lambda},B(0,r),0)=
		\begin{cases}
			\phantom{-}1 &\text{ if }\lambda<\lambda_{1}(s,p),\\
		   -1 &\text{ if }\lambda_{1}(s,p)<\lambda
		   <\lambda_{2}(s,p).
		\end{cases}
 		\]
\end{teo}

This result is a generalization of Proposition 2.2 in \cite{delPino},
where the authors show that
\begin{equation}\label{eq:caso1p}
	\deg_{\wupc}(I-\T_{1,p}^{\lambda}(u),B(0,r),0)=
		\begin{cases}
			\phantom{-}1 &\text{ if }\lambda<\lambda_{1}(1,p),\\
		   -1 &\text{ if }\lambda_{1}(1,p)<\lambda<\lambda_{2}(1,p).
		\end{cases}
\end{equation}

\begin{proof}
	Let $s_0 \in(0,1].$ We begin by the case 
	$\lambda_1(s_0,p)<\lambda<\lambda_2(s_0,p).$
	By Lemmas \ref{lem:eigcont} and \ref{lem:lgi}, there exists a
	continuous function $\rho\colon[s_0,1]\to\R$ such that
	\[
		\lambda_1(s,p)<\rho(s)<\lambda_2(s,p)\quad \forall s\in[s_0,1],
	\] 
	and $\rho(s)=\lambda.$
	Then it is sufficient to prove that the function
	$d\colon[s_0,1]\to\R$
	\[
		d(s)\coloneqq
		\deg_{\twsp}(I-\T_{s,p}^{\rho(s)},B(0,r),0)
	\]
	is constant due to $d(1)=-1$
	
	Let $s\in[s_0,1].$ We define the operator 
	$\mathcal{P}_s\colon \lp\to\twsp$ as 
	\[
		\mathcal{P}_s(u)\coloneqq\RR_{s,p}(\rho(s)|u|^{p-2}u).
	\] 
	Then $\mathcal{P}_s$ is completely continuous and
	\[
		\T_{s,p}^{\rho(s)}=\mathcal{P}_s\circ i
	\]
	where $i\colon\twsp\to\lp$ is the usual inclusion. Thus, by
	Lemma \ref{lem:LSDP}, we get
	\begin{equation}\label{eq:246}
		d(s)=\deg_{\lp}(I-i\circ\mathcal{P}_s,O,0)\quad\forall s\in[s_0,1]
	\end{equation}
	where $O$ is any open bounded set in $\lp$ such that $0\in O.$ 
	
	On the other hand, since $\rho$ is continuous and by Lemma 
	\ref{lem:eigcont}, we get that the homotopy
	\[
		[s_0,1]\times\lp\to\lp
	\]
	\[
		(s,u)\to\RR_{s,p}(\rho(s)|u|^{p-2}u)=(i\circ\mathcal{P}_s)(u)
	\]
	is completely continuous. Then $d(s)$ is constant in $[s_0,1]$
	due to the invariance of the Leray-Schauder degree under compact
	homotopy and \eqref{eq:246}.
	
	\medskip
	
	Finally, we consider the case $\lambda<\lambda_{1}(s_0,p).$
	Given $a\in[0,1],$ the degree
	\[
		\deg_{\widetilde{W}^{s_0,p}(\Omega)}
		(I-\RR_{s_0,p}(a\lambda \Psi_p(\cdot)),B(0,r),0)
	\]
	is well defined. Here $\Psi_p(u)=|u|^{p-2}u.$
	Then, from the invariance of the degree under homotopies, we get
	\[
		\deg_{\widetilde{W}^{s_0,p}(\Omega)}
		(I-\RR_{s_0,p}(a\lambda \Psi_p(\cdot)),B(0,r),0)=
		\deg_{\widetilde{W}^{s_0,p}(\Omega)}
		(I,B(0,r),0)=1
	\]
	for all $a\in[0,1].$
\qed \end{proof}

Finally, proceeding as in the proof of Theorem 1.1 in \cite{delPino}, we
prove Theorem \ref{teo:teo11}. 

\begin{proof}[Proof of Theorem \ref{teo:teo11}]
	By contrary, suppose that $(\lambda_1(s,p),0)$ is no  bifurcation point of
	\eqref{eq:elproblema}. Then there exist $\varepsilon,\delta_0>0$
	such that for
	$|\lambda-\lambda_1(s,p)|\le\varepsilon$ and $\delta<\delta_0$
	there is no  non-trivial solution of
	\[
		u-\mathscr{R}_{\lambda}(u)=0
	\]
	with $\|u\|_{\wspr}=\delta.$ Since the degree is invariance under
	compact homotopies
	\begin{equation}\label{eq:degree}
		\deg_{\twsp}
		(I-\mathscr{R}_{\lambda},B(0,\delta),0)=
		\text{constant} 
	\end{equation}
	for all $\lambda\in[\lambda_1(s,p)-\varepsilon,
	\lambda_1(s,p)+\varepsilon].$
	
	Taking $\varepsilon$ small enough, we can assume that there is no
	eigenvalue of \eqref{eq:paut1} 
	in $(\lambda_1(s,p),\lambda_1(s,p)+\varepsilon].$
	
	Fix now $\lambda \in (\lambda_1(s,p),\lambda_1(s,p)+\varepsilon].$
	We claim that if choose $\delta$ small enough then there is no solution
	of 
	\[
		u-\RR_{s,p}(\lambda|u|^{p-2}u+tF(\lambda,u))=0
	\]
	with $\|u\|_{\wspr}=\delta,$ for all $t\in[0,1].$
	Indeed, assuming the contrary and reasoning as in the proof of Proposition 2.1 of \cite{delPino} (see also Lemma \ref{lem:bifaut}), we would find that $\lambda$ is an eigenvalue of \eqref{eq:paut1}, that is a contradiction.
	
	Thus, since the degree is invariance under homotopies, 
	by Theorem \ref{teo:gradoT}, 
	\[
		\deg_{\widetilde{W}^{s,p}(\Omega)}
		(I-\mathscr{R}_{\lambda},B(0,\delta),0)=
		\deg_{\widetilde{W}^{s,p}(\Omega)}
		(I-\T_{s,p}^{\lambda},B(0,\delta),0)=-1.
	\] 
	
	In similar manner, we can see that
	\[
		\deg_{\widetilde{W}^{s,p}(\Omega)}
		(I-\mathscr{R}_{\lambda},B(0,\delta),0)=1
	\]
	for all $\lambda\in[\lambda_1(s,p)-\varepsilon,\lambda_1(s,p)).$
	Therefore  $\deg_{\widetilde{W}^{s,p}(\Omega)}
	(I-\mathscr{R}_{\lambda},B(0,\delta),0)$ is no constant function.
	But this is a contradiction with \eqref{eq:degree} and so 
	$(\lambda_1(s,p),0)$ is a bifurcation point of
	\eqref{eq:elproblema}. 
	
	The rest of the proof follows in the same manner as in \cite{RA2}. \qed
\end{proof}

\section{Existence of constant-sign solution}\label{aplica}
Let $\Omega\subset\R^n$ be a bounded domain with Lipschitz boundary,
$s\in(0,1),$ $p\in(1,\infty),$ and $g\colon\R\to\R$ be a continuous function
such that $g(0)=0.$ In this section, we will apply Theorem \ref{teo:teo11} to show that the following non-linear 
non-local problem
\begin{equation}\label{eq:D}
	\begin{cases}
		(-\Delta)^s_p u=g(u) &\text{ in } \Omega,\\
		u=0 &\text{ in } \R^n\setminus \Omega,\\
	\end{cases}
\end{equation} 
has a non-trivial weak solution. Observe that $u\equiv0$ is a solution of 
\eqref{eq:D}.

\medskip

We will keep the following assumptions about $g$, throughout this section:
\begin{enumerate}[\mbox{A}1.]
\itemsep1.5em 
\item	$\dfrac{g(t)}{|t|^{p-2}t}$ is bounded;
\item  $\underline{\lambda}\coloneqq\displaystyle\lim_{t\to0}
		{\dfrac{g(t)}{|t|^{p-2}t}<\lambda_1(s,p)<
		\displaystyle\liminf_{|t|\to\infty}\dfrac{g(t)}{|t|^{p-2}t}}.$ 
\end{enumerate}

\medskip

Note that, if $g$ satisfies A1 and A2 then
\[
	g(t)=\underline{\lambda}|t|^{p-2}t + f(t),
\]
where $f(t)=o(|t|^{p-1})$ near $t=0.$ Then, our problem is related to the next 
bifurcation problem
\begin{equation}\label{eq:B}
	\begin{cases}
		(-\Delta)^s_p u=\lambda |u|^{p-2}u+ f(u) &\text{ in } \Omega,\\
		u=0 &\text{ in } \R^n\setminus \Omega.
	\end{cases}
\end{equation} 

By Theorem \ref{teo:teo11} there exists a connected component $\CC$ of the
set of non-trivial solution of \eqref{eq:B} in
$\R\times\twsp$  whose closure contains $(\lambda_1(s,p),0)$
and it is either unbounded or contains a pair $(\lambda,0)$
for some $\lambda,$ eigenvalue of \eqref{eq:ep} with
$\lambda>\lambda_1(s,p).$

\begin{lem}\label{lem:lema31}
	Let $\Omega\subset\R^n$ be a bounded domain with Lipschitz boundary,
	$s\in(0,1),$ and $p\in(1,\infty).$ Then $\CC$ is unbounded and
	\[
		\overline{\CC}\subset
		\HH\coloneqq\{(\lambda_1(s,p),0)\}\cup(\R\times\PP),
	\] 
	where $\PP\coloneqq\{v\in\twsp\colon v
	\mbox{ has constant-sign in } \Omega\}.$
\end{lem}
\begin{proof}
	We split the proof in 3 steps.
	
	\medskip
	
	\noindent {\it Step 1.} There exists a neighbourhood $U$ of 
	$(\lambda_1(s,p),0)$ in $\R\times\twsp$ such that
	$\CC\cap U\setminus\{(\lambda_1(s,p),0)\}\subset\R\times\PP.$
	
	Let us assume by contradiction the existence of a sequence 
	$\{(\lambda_k,u_k)\}_{k\in\N}$ of non-trivial solution of \eqref{eq:B}
	such that $u_k$ changes sign in $\Omega$ for all $k\in\N$ and
	$(\lambda_k,u_k)\to (\lambda_1(s,p),0)$ in $\R\times\twsp$ as
	$k\to\infty.$ 
	
	For any $k\in\N,$ since $h_k=\lambda_k + \nicefrac{f(u_k)}{|u_k|^{p-2}u_k}$
	is uniformly bounded in $\Omega$ and $u_k$ changes sign in $\Omega,$
	by Corollary \ref{co:bf2} we have that $1$ is an eigenvalue
	\eqref{eq:epf} with $h=h_k$ and $1>\lambda_1(s,p,h_k).$ Thus, by Lemma
	\ref{lem:autoval3} and using that $h_k$ is uniformly bounded in $\Omega$, 
	there exists a constant $C$ independent of $k$ such that
	\begin{equation}\label{eq:auxap1}
			|\{x\in\Omega\colon u_k(x)>0\}|\ge C 
			\mbox{ and }
			|\{x\in\Omega\colon u_k(x)<0\}|\ge C \quad\forall k\in\N.
	\end{equation} 
	
	On the other hand, taking 
	$\hat{u}_k\coloneqq\nicefrac{u_k}{\|u_k\|_{\twsp}},$ it follows that
	the sequence $\{\hat{u}_k\}_{k\in\mathbb{N}}$ is bounded in $\twsp$ then,
	via a subsequence if necessary, we have that there exists $u\in\twsp$ such 
	that 
	\begin{align*}
		\hat{u}_k\rightharpoonup u &\mbox{ weakly in } \twsp,\\
		\hat{u}_k\to u &\mbox{ strongly in } \lp,\\
		\hat{u}_k\to u &\mbox{ a.e. in } \Omega.
	\end{align*}
	By \eqref{eq:auxap1}, 
	\begin{equation}\label{eq:auxap2}
	  u\not\equiv0 \mbox{ and } u \mbox{ changes sign.}
	\end{equation}
	Moreover
	\begin{align*}
		\K(1-s)|u|_{\wspr}^p
		&\le 
		\lim_{k\to\infty}\K(1-s)\int_{\R^n}\!\!
		\int_{\R^n}\!\!\!\!
		\dfrac{|\hat{u}_k(x)-\hat{u}_k(y)|^p}{|x-y|^{n+sp}} dx dy\\
		&= 
		\lim_{k\to\infty}\int_{\Omega}h_k(x)|\hat{u}_k|^p\, dx\\
		&=\lambda_1(s,p)\int_\Omega |u|^p dx,
	\end{align*}
	due to $h_k$ is uniformly bounded in $\Omega,$ $h_k(x)\to\lambda_1(s,p)$
	a.e. in $\Omega$ and $\hat{u}_k\to u$ strongly in $\lp.$ Then
	\[
		\K(1-s)\dfrac{
		\displaystyle\int_{\R^n}\int_{R^n}\dfrac{|u(x)-u(y)|^p}{|x-y|^{n+sp}} 
		dx dy}{\displaystyle\int_\Omega |u|^p dx }\le\lambda_1(s,p).
	\]
	Thus, by definition of $\lambda_1(s,p),$ we have that 
	$u$ is an eigenfunction 
	associated to $\lambda_1(s,p).$ Therefore, by Corollary \ref{co:bf2}, 
	$u$ has constant sign, this yield a contradiction with
	\eqref{eq:auxap2}. Hence the claim follows.
	
	\medskip
	
	\noindent {\it Step 2.} $\overline{\CC}\subset\HH.$
	
	Again we proceed by contradiction. Suppose that there exists 
	$(\lambda_0,u_0)\in\overline{\CC}$ such that can be approximated by 
	elements of $\CC$ from inside and from without $\HH,$ that is 
	there exist $\{(\lambda_k,u_k)\}_{k\in\n}\subset \CC\cap\HH$ 
	and $\{(\mu_k,v_k)\}_{k\in\n}\subset \CC\cap\HH^c$ such that
	$(\lambda_k,u_k)\to(\lambda_0,u_0) $ and $(\mu_k,v_k)\to(\lambda_0,u_0). $
	By step 1, $(\lambda_0,u_0)\neq(\lambda_1(s,p),0).$
	
	\medskip
	
	Case $u_0\equiv 0.$ Thus $\lambda_0\neq\lambda_1(s,p).$  
	Proceeding in a similar manner as in the previous step, 
	we can see that $\lambda_0$ is an eigenvalue of \eqref{eq:paut1} 
	 different to $\lambda_1(s,p)$ and arrive to a contradiction.
	
	\medskip
	
	Case $u_0\not\equiv 0.$ We know that there exist 
	$\{(\lambda_k,u_k)\}_{k\in\N}\subset\CC\cap\HH$ such that
	$(\lambda_k,u_k)\to (\lambda_0,u_0)$ in $\R\times\twsp.$ Therefore
	$u_0$ is either non-negative or non-positive and $u_0$ is a weak solution of
	\[
		\begin{cases}
			(-\Delta)^s_p 
			u=\left(\lambda_0 + \dfrac{f(u_0)}{|u_0|^{p-2}u_0}
			\right)|u|^{p-2}u &\text{ in } \Omega,\\
			u=0 &\text{ in } \R^n\setminus \Omega.
		\end{cases}
	\]
	Without loss of generality, we can assume that $u_0\ge0$ a.e. in 
	$\Omega.$ Since 
	$\lambda_0 + \nicefrac{f(u_0)}{|u_0|^{p-2}u_0}$ is bounded,
	by Theorem \ref{teo:stm}, we have that $u_0>0$ a.e. in $\Omega.$ Thus
	$(\lambda_0,u_0)\in\HH.$  
	Argument in similar manner that in step 1, we can show that
	$(\lambda_0,u_0)$ can not be approximated by elements of $\CC$ from
	without $\HH,$ contradicting the fact that 
	$(\lambda_0,u_0)$ can be approximated by 
	elements of $\CC$  from without $\HH.$
	
	\medskip
	
	\noindent {\it Step 3.} $\CC$ is unbounded.
	
	Since $\overline{\CC}\subset\HH,$ $\CC$ does  not  contain a 
	pair $(\lambda,0)$ for some $\lambda,$ eigenvalue of \eqref{eq:ep} with
	$\lambda>\lambda_1(s,p).$ Then by Theorem \ref{teo:teo11}, $\CC$ is 
	unbounded.
\qed \end{proof}

Our next aim is to show that $\CC\cap\left(
[\underline{\lambda},\infty)\times\twsp\right)$ is bounded. 
For this, we will need the following result. The proof is identical to the proof
of \cite[Lemma 3.2]{delPino}.

\begin{lem}\label{lema:cotaaut} 
	There exists a positive constant $C$ such that
	if $(\lambda,u)\in\CC$ then $\lambda\le C.$
\end{lem}

Then for showing that $\CC\cap\left(
[\underline{\lambda},\infty)\times\twsp\right)$ is bounded, it is 
enough to prove the result given below.

\begin{lem}\label{lema:cota}
	There exists a positive constant $M$ such that for any $(\lambda,u)\in
	\CC\cap\left([\underline{\lambda},C]\times\twsp\right)$
	we have that $\|u\|_{\twsp}\le M.$ Here $C$ is the constant of Lemma
	\ref{lema:cotaaut}.
\end{lem}

\begin{proof} 
	Suppose by contradiction that there exists a sequence 
	$\{(\lambda_k,u_k)\}_{k\in\mathbb{N}}$ of elements of 
	$\CC\cap\left([\underline{\lambda},C]\times\twsp\right)$
	such that $\lambda_k\to\lambda_0$ and 
	$\|u_k\|_{\twsp}\to\infty$ as $k\to\infty.$ Without loss of generality
	we can assume that $u_k>0$ for all $k\in\N.$
	
	Taking $\hat{u}_k=\nicefrac{u_k}{\|u_k\|_{\twsp}}$ and
	$h_k=\nicefrac{f(u_k)}{|u_k|^{p-2}u_k},$ for any $k\in\N$ we have that
	\[
		\hat{u}_k=\RR_{s,p}
		\left(\lambda_k|\hat{u}_k|^{p-2}\hat{u}_k+
		\frac{f(u_k)}{|u_k|^{p-2}u_k}
		|\hat{u}_k|^{p-2}\hat{u}_k
		\right).
	\]
	
	On the other hand, $\{\nicefrac{f(u_k)}{|u_k|^{p-2}u_k}\}_{k\in\N}$ 
	is uniformly bounded
	due to $g$ satisfies A1, then there exists $h\in L^\infty(\Omega)$ 
	such that
	\[
		\dfrac{f({u}_k)}{|u_k|^{p-2}u_k}\rightharpoonup h 
		\mbox{ weakly in } L^q(\Omega) \quad\forall q>1.
	\]
	
	Since $\RR_{s,p}$ to $L^{q^\prime}(\Omega)$ with 
	$q\in(1, p^\star_s)$ is a completely continuous operator, we have that
	there exists $u_0\in\twsp$ such that $u_k\to u_0$ strongly in 
	$\twsp$ and 
	\[
		u_0=\RR_{s,p}
		\left(\lambda_0|u_0|^{p-2}u_0+
		h|u_0|^{p-2}u_0
		\right),
	\]
	that is $u_0$ is a weak solution of 
	\[
		\begin{cases}
			(-\Delta)^s_p u=\left(\lambda_0 + h(x)
			\right)|u|^{p-2}u &\text{ in } \Omega,\\
			u=0 &\text{ in } \R^n\setminus \Omega.
		\end{cases}
	\]
	Observe that $u_0\neq0$ and $u\ge0$ due to $\|\hat{u}_k\|_{\twsp}=1$
	and $\hat{u}_k>0$ in $\Omega.$ Hence $\mu=1$ is the first eigenvalue of
	\begin{equation}\label{eq:auxteofin}
		\begin{cases}
			(-\Delta)^s_p u=\mu\left(\lambda_0 + h(x)
			\right)|u|^{p-2}u &\text{ in } \Omega,\\
			u=0 &\text{ in } \R^n\setminus \Omega,
		\end{cases}
	\end{equation}
	and $u_0$ is an eigenfunction associated to $1.$ Then, 
	by Corollary \ref{co:bf2}, we have that $u_0>0$ in $\Omega.$
	
	\medskip
	
	\noindent{\it Claim.} 
	$h\ge \overline{\lambda}- \underline{\lambda}$ a.e. in 
	$\Omega$ where $\lambda_1(s,p)<\overline{\lambda}<\displaystyle
	\liminf_{|s|\to\infty}\dfrac{g(s)}{|s|^{p-2}s}.$ 
	
	Suppose the contrary, that is the set $A=\{x\in\Omega\colon h(x)<
	\overline{\lambda}- \underline{\lambda}\}$ has positive measure. 
	Since $\hat{u}_k\to u_0>0$ a.e. in $\Omega,$ by the Egorov's theorem,
	there exists a set $U\subset\Omega$ such that $|\Omega\setminus U|<|A|$
	and $u_k\to\infty$  uniformly in $U.$ Then there exists $k_0\in\N$ such
	that $\nicefrac{f(u_k)}{|u_k|^{p-2}u_k}\ge
	\overline{\lambda}- \underline{\lambda}$ for all $k\ge k_0$ 
	because 
	\[
	\lambda_1(s,p)<\overline{\lambda}<\displaystyle
	\liminf_{|s|\to\infty}\dfrac{g(s)}{|s|^{p-2}s}=
	\underline{\lambda}+\liminf_{|s|\to\infty}\dfrac{f	(s)}{|s|^{p-2}s}
	\]
	and 
	therefore $h(x)\ge\overline{\lambda}- \underline{\lambda} $ a.e. in
	$U.$ Thus $A\subset \Omega\setminus U,$ then 
	$|A|\le|\Omega\setminus U|<|A|,$ which is a contradiction. Hence, 
	the claim follows.
	
	\medskip
	
	Since $h(x)\ge \overline{\lambda}- \underline{\lambda}$ a.e. in 
	$\Omega,$ $\lambda_0
	-\underline{\lambda}\ge 0$ and $\overline{\lambda}>\lambda_1(s,p),$ 
	we get $\lambda_0 + h(x)\ge 
	\lambda_0 + \overline{\lambda}- \underline{\lambda}>\lambda_1(s,p).$
	
	On the other hand, 
	since $\mu$ is the first eigenvalue of \eqref{eq:auxteofin},
	we have that
	\[
		1\le\K(1-s)\dfrac{|\phi|_{\wspr}^p}{\displaystyle\int_\Omega 
		(\lambda_0+h(x))|\phi(x)|^p\, dx}\quad \forall\phi\in 
		C_0^\infty(\Omega).
	\] 
	Then for any $\phi\in C_0^\infty(\Omega)$
	\[
		(\lambda_0 + \overline{\lambda}- \underline{\lambda})
		\|\phi\|_{\lp}^p\le\int_\Omega 
		(\lambda_0+h(x))|\phi(x)|^p\, dx\le \K(1-s)|\phi|_{\wspr}^p
	\]
	due to our claim. Then
	\[
		\lambda_0 + \overline{\lambda}- \underline{\lambda}
		\le\lambda_1(s,p)<\lambda_0 + 
		\overline{\lambda}- \underline{\lambda},
	\]
	getting a contradiction. Thus the lemma is true.\qed 
	
\end{proof}

Finally, we prove Theorem \ref{teo:teo31}.

\begin{proof}[Theorem \ref{teo:teo31}]
	By Lemma \ref{lema:cotaaut} and Lemma \ref{lema:cota}, 
	$\CC\cap\left([\underline{\lambda},\infty)\times\twsp\right)$ is 
	bounded. On other hand, by Lemma \ref{lem:lema31}, $\CC$ is unbounded. Then
	there exists $(\underline{\lambda},\underline{u})\in\CC,$ due to $\CC$
	is connected. By A2 $\underline{\lambda}<\lambda_1(s,p)$ and 
	Lemma \ref{lem:lema31}, $\underline{u}$ has constant-sign in $\Omega.$
	Therefore $u$ is a non-trivial weak solution of \eqref{eq:D}.
\qed \end{proof}

\subsection*{Acknowledgements}
This work was partially supported by Mathamsud project 
13MATH--03 -- QUESP -- Quasilinear Equations and Singular Problems.
L. M. Del Pezzo was partially supported by PICT2012 0153 from ANPCyT (Argentina) and A. Quaas was partially supported 
by Fondecyt Grant
No. 1151180 Programa Basal, CMM. U. de Chile and Millennium Nucleus
Center for Analysis of PDE NC130017.

\bibliographystyle{amsplain}
\bibliography{Biblio}

\end{document}